\documentclass{article}
\usepackage{latexsym}
\usepackage{leftidx}
\usepackage{slashed}
\usepackage{amsmath,amsthm, pdfpages}
\usepackage{enumerate}
\usepackage{amssymb}
\usepackage{upgreek}

\usepackage[margin=1.3in,dvips]{geometry}
\def\I{\mathcal{I}}

\def\nn{\nonumber}

\def\Lie{{\cal L}}

\newcommand{\vol}{\textnormal{vol}}

\def\f12{\frac 1 2}

\def\a{\alpha}
\def\b{\beta}
\def\ga{\gamma}
\def\Ga{\Gamma}

\def\ep{\epsilon}

\def\si{\sigma}
\def\Si{\Sigma}
\def\om{\omega}
\def\Om{\Omega}

\def\cL{\mathcal{L}}

\def\H{\mathcal{H}}
\def\Hb{\underline{\mathcal{H}}}
\def\E{\mathcal{E}}%energy

\def\ab{\underline{\a}}
\def\Lb{\underline{L}}

\def\nab{\nabla}
\def\pa{\partial}
\def\les{\lesssim}

\def\G{\mathcal{G}}
\def\sl{\slashed}
\def\ud{\underline}
\def\dn{{\Delta_0}}
\def\pr{\partial}
\newcommand{\D}{\mbox{$D \mkern-13mu /$\,}}

\newcommand{\nabb}{\mbox{$\nabla \mkern-13mu /$\,}}
\newcommand{\J}{\mbox{$J \mkern-13mu /$\,}}

\newtheorem{thm}{Theorem}

\newtheorem{prop}{Proposition}

\newtheorem{lem}{Lemma}
\newtheorem{cor}{Corollary}
\newtheorem{remark}{Remark}

\begin{document}

\title{Global solution for massive Maxwell-Klein-Gordon equations}
\author{Sergiu Klainerman \and Qian Wang \and Shiwu Yang}
\AtEndDocument{\bigskip{\footnotesize%
  \addvspace{\medskipamount}
  \textsc{\quad \ \ Department of Mathematics, Princeton University, Princeton, New Jersey, USA} \par
  \textit{E-mail address}: \texttt{seri@math.princeton.edu} \par
  \addvspace{\medskipamount}
  \textsc{Oxford PDE Center, Mathematical Institute, University of Oxford, Oxford, OX2 6GG, UK} \par
  \textit{E-mail address}: \texttt{qian.wang@@maths.ox.ac.uk}\par
  \addvspace{\medskipamount}
  \textsc{Beijing International Center for Mathematical Research, Peking University, Beijing, China} \par
  \textit{E-mail address}: \texttt{shiwuyang@math.pku.edu.cn}
}}
%\address{
%Oxford PDE center, Mathematical Institute, University of Oxford, Oxford, OX2 6GG, UK}
% \email{qian.wang@maths.ox.ac.uk}
\date{}

\maketitle

\begin{abstract}
We derive the asymptotic properties of the mMKG system (Maxwell  coupled with a massive Klein-Gordon scalar field), in the exterior
  of  the domain of influence of a compact set.  This complements the previous well known results,
  restricted to  compactly supported  initial conditions,  based on the so called hyperboloidal method.  That method takes advantage of  the  commutation properties of the   Maxwell and Klein Gordon  with the generators of the Poincar\'e group to  resolve the
   difficulties    caused by the fact that they have,  separately,    different   asymptotic properties.     Though the hyperboloidal  method is very robust and applies well
       to  other  related systems it has  the  well known  drawback  that it    requires compactly supported data.  In this paper we remove  this limitation
      based on a further extension of the vector-field  method adapted to the exterior region.  Our method  applies, in particular, to nontrivial charges.   The full problem could then be treated by patching together  the new  estimates in the  exterior  with  the hyperboloidal ones  in the interior.    This   purely physical space  approach  introduced here   maintains the  robust properties of the
      old method and can thus be applied to other situations such as    the coupled Einstein Klein-Gordon equation.

\end{abstract}

\section{Introduction}
In this paper, we study the small data global solutions to the massive Maxwell-Klein-Gordon equations on $\mathbb{R}^{3+1}$. To define the equations, let $A=A_\mu dx^\mu$ be a $1$-form. The covariant derivative associated to this 1-form is
\begin{equation*}
 %\label{defofD}
D_\mu =\pa_\mu+\sqrt{-1}A_\mu,
\end{equation*}
which can be viewed as a $U(1)$ connection on the complex line bundle over $\mathbb{R}^{3+1}$ with the standard flat metric $m_{\mu\nu}$. Then the curvature $2$-form $F$ is given by
\begin{equation*}
 %\label{defofF}
F_{\mu\nu}=-\sqrt{-1}[D_{\mu}, D_{\nu}]=\pa_\mu A_\nu-\pa_\nu A_\mu=(dA)_{\mu\nu}.
\end{equation*}
This is a closed $2$-form or equivalently $F$ satisfies the Bianchi identity
\begin{equation}
\label{bianchi}
 \pa_\ga F_{\mu\nu}+\pa_\mu F_{\nu\ga}+\pa_\nu F_{\ga\mu}=0.
\end{equation}
Note also that,
\begin{equation}\label{1.17.1.18}
(D_\mu D_\nu -D_\nu D_\mu)\phi=\sqrt{-1} F_{\mu\nu}\phi.
\end{equation}
The massive Maxwell-Klein-Gordon equations (mMKG) is a system for the connection field $A$ and the complex scalar field $\phi$:
\begin{equation}
 \label{EQMKG}\tag{mMKG}
\begin{cases}
\pa^\nu F_{\mu\nu}=\Im(\phi \cdot\overline{D_\mu\phi})=J[\phi]_\mu;\\
D^\mu D_\mu\phi-\phi=\Box_A\phi-\phi=0,
\end{cases}
\end{equation}
in which the mass of the scalar field is normalized to be 1.
%These are Euler-Lagrange equations of the functional
%\[
%L[A, \phi]=\iint_{\mathbb{R}^{3+1}}\frac{1}{4}F_{\mu\nu}F^{\mu\nu}+\frac{1}{2}D_{\mu}\phi\overline{D^{\mu}\phi}+\f12 |\phi|^2dxdt.
%\]
It is well known that this system is gauge invariant in the sense that $(A-d\chi, e^{i\chi}\phi)$
%\[
% \phi\mapsto e^{i\chi}\phi; \quad A\mapsto A-d\chi.
%\]
solves the same equation \eqref{EQMKG} for any potential function $\chi$.

In this paper, we consider the Cauchy problem to \eqref{EQMKG} with initial data $(\phi_0, \phi_1, E, H)$:
\begin{equation*}
% \label{IDset}
F_{0i}(0, x)=E_i(x),\quad \leftidx{^*}F_{0i}(0, x)=H_i(x),\quad \phi(0, x)=\phi_0,\quad D_0\phi(0, x)=\phi_1,
\end{equation*}
where $\leftidx{^*}F$ is the Hodge dual of the 2-form $F$. The (mMKG) equation system is an over-determined system which means that
%In local coordinates $(t, x)$,
%\[
%(H_1, H_2, H_3)=(F_{23}, F_{31}, F_{12}).
%\]
 the data set has to satisfy the compatibility condition
\begin{equation*}
div(E)=\Im(\phi_0\cdot \overline{\phi_1}),\quad div (H)=0
\end{equation*}
with the divergence taken on the initial hypersurface $\mathbb{R}^3$.
For solutions of \eqref{EQMKG}, the total energy
%\[
% E[F, \phi](t):=\int_{\mathbb{R}^3}(|E|^2+|H|^2+|D\phi|^2+|\phi|^2)dx
%\]
as well as the total charge
\begin{equation}
\label{defcharge}
q_0=\frac{1}{4\pi}\int_{\mathbb{R}^3}\Im(\phi\cdot \overline{D_0\phi})dx=\frac{1}{4\pi}\int_{\mathbb{R}^3}div (E)dx
\end{equation}
are conserved. The existence of nonzero charge has a long range effect on the asymptotic behaviour of the solution as the electric field $E_i=F_{0i}$ has a tail $q_0r^{-3}x_i$ at any fixed time $t$.

\medskip

The Maxwell-Klein-Gordon system \eqref{EQMKG}     has  drawn extensive attention in the past. The pioneering works \cite{Moncrief1} and \cite{Moncrief2} of Eardley-Moncrief established a  global existence result to the more general Yang-Mills-Higgs equations with sufficiently smooth initial data in $\mathbb{R}^{3+1}$. For the massless case, when   the scalar field   has the same radiation properties as  the Maxwell field,  Klainerman-Machedon \cite{MKGkl}  obtained a  global  existence and regularity result  in  the energy space by introducing the celebrated bilinear estimates for null forms \cite{Kl-Ma1}. These type of estimates,   as well as  further  improvements made by various authors,  have     revolutionized  our understanding      of   optimal  well-posedness  for  classical field equations, such as Maxwell-Klein-Gordon, Yang-Mills, Wave Maps,  Einstein Field Equations etc,  whose ultimate  aim is to   provide  an effective  continuation argument\footnote{This goal has been  met successfully  for semilinear fields such as  Wave Maps, Maxwell-Klein-Gordon and Yang-Mills, but remains a major open problem for quasilinear ones  such as the Einstein  equations.}   consistent with the scaling properties of the equations, see \cite{Kl-1999} for a general discussion of the topic.
For such works, in the context of  massless  MKG,  we refer to \cite{YMkl}, \cite{MKGigor}, \cite{KriegerMKG4}, \cite{MKGtataru}, \cite{OhMKG4} and references therein.

\medskip

The long time asymptotic  behavior  of smooth  solutions  to the massless MKG equations  in  $\mathbb{R}^{3+1}$ is  also relatively   well understood. The vector field method introduced by Klainerman \cite{klinvar} is sufficiently robust to obtain the pointwise decay estimates for small solutions  with non-trivial charges, see e.g. \cite{LindbladMKG}.
 For the restricted  case of compactly supported initial data, one can also use the conformal compactification method\footnote{The conformal compactification method does not  in fact require compactly supported data but it imposes serious restriction on the rate of decay at infinity.  It requires in particular that $E=O(|x|^{-4})$ and thus excludes non-trivial charges.} \cite{fieldschrist}, \cite{Deivy:decayMKG:trivialout}.
 By combining  the Eardley-Moncrief result  with  the conformal method one can also  obtain  similar results for large compactly supported data \cite{Br-Seg}.   For  large initial  conditions,    with  nontrivial charges,   the first result  establishing
 quantitative energy flux decay of the solutions  are due to  S. Yang in  \cite{yangILEMKG}.  Finally,
  by combining the modified vector field method  of \cite{yangILEMKG},   \cite{yangMKG}    with the conformal compactification method,
  Yang-Yu in \cite{YangYu:MKGrough}
 and    \cite{YangYu:MKG}   have recently given a full description of the global asymptotic dynamics of massless MKG  equations,  with  large initial data and nontrivial charges,  in $\mathbb{R}^{3+1}$.

\medskip

The  decay properties of the  massive Maxwell-Klein-Gordon (mMKG)  equations  are far  less understood.
 In the case  of  an  uncoupled, nonlinear scalar  Klein Gordon equation \cite{Kl:KG:85},  Klainerman found a variation  of the vectorfield method, based on  the standard   hyperboloid foliation  of the interior of a  forward light  cone, which allows one to derive\footnote{A similar result  based on Fourier methods  and renormalization   is due to   J. Shatah
 \cite{shatah}.  Note that  \cite{shatah}  does not require    restrictions on the data. } global existence and    asymptotic behavior  of solutions corresponding to small, compactly supported  data. The hyperboloidal method can be easily  extended to  the
  mMKG equations provided that one maintains  the restriction   of the data, see\footnote{See also  \cite{Maria:timedecay:mMKG} for an integrated local energy estimate.  }   \cite{Maria:mMKG:small}.  The goal of this paper is to describe a new variation of the vectorfield  method  which allows us to  dispense of the  restriction to compactly supported data. We note that, in principle, the Fourier type methods recently  introduced in \cite{Ionescu:waveKG}  may also apply to the system considered here to derive comparable results.   We  believe  however  that the  purely geometric, physical space,  approach   taken here    has its own specific advantages  which     will prove to be of independent interest in various  applications.

The main difference between the massive   and  massless cases is due to the different global  asymptotic behavior  of solutions to  wave  and Klein-Gordon equations.
Recall that   solutions of the standard  wave equation in $\mathbb{R}^{3+1}$  concentrate   along outgoing null directions  and decay faster in the interior  while, to the contrary,  solutions
to the standard KG equation   concentrate in the interior and decay faster along null  directions.
This  difference can be neatly captured  by the  vectorfield method. Indeed while  both   $\square$ and   $\square-1$  commute   with the generators of the   Lorentz group $\Om_{\a\b}=x_\a\pr_\b-x_\b \pr_\a$,  it is only  the  former    which    has, in addition,   good commutation properties  with the   scaling vectorfield  $S= x^\a \pr_\a$
 as well as the  other conformal Killing vector fields of Minkowski space.  The method pioneered by  Klainerman in  \cite{Kl:KG:85} takes advantage  of the commutation properties
 of $\square-1$  with $\Om_{\a\b}$      by           considering   higher order, invariant,   energy estimates  along   the natural hyperboloidal   foliation, asymptotic to a fixed outgoing null cone,    spanned by these  vector fields.    The decay properties of solutions then follow by a  simple global Sobolev inequality on the  leaves of the foliation. The method
 has the advantage that it can be applied to  coupled nonlinear   wave and Klein-Gordon type equations  and, indeed,  it  has been later  widely used   in various such applications, most notable the recent works of \cite{Mayue:EKG} and \cite{Qian:EKG} on the coupled Einstein-Klein-Gordon equations, see also  \cite{Maria:mMKG:small} and \cite{Katayama:waveKl:12}.

 \medskip

 The obvious  limitation of the hyperboloidal approach is  that it  is naturally restricted to solutions supported in a forward light cone, i.e. for compactly\footnote{ As far as we know, the general case  can only be treated by   Fourier techniques, see \cite{shatah} for nonlinear Klein-Gordon equations  and see \cite{Ionescu:waveKG} for  wave-KG  coupled systems.}  supported initial  data.  The main contribution of this paper is to   provide     the missing piece,  i.e. to   construct
 solutions  to mMKG system in the complement of     a fixed light cone.         The full problem could then be treated by patching together  the new  estimates in the  exterior  with  the hyperboloidal ones  in the interior.
 \medskip

   A first glimpse of  what is needed to  derive the decay of solutions  to  KG  equations  in  the exterior of a light cone can be found in  \cite{Kl:remark:KG} where Klainerman gave a hierarchic multiplier-vectorfield   approach to recover the fast decay  along null directions.  Unfortunately  that approach  does not directly apply  to    the  Maxwell equation  piece of the mMKG system, and  we cannot rely on either  commutation with the  scaling vectorfield  or the hyperboloidal foliation.
   The  restriction  to the exterior of  an outgoing null  cone provides  however a  small but  crucial advantage in that we  can derive   energy-flux  estimates  with higher weights  in $|u|$  \big(the outgoing optical function $u=\frac 1 2 (t-r)\big) $ in such regions,  see  e.g. \cite{KlN2},  and that such estimates also apply to KG.         Combining these two ingredients is  not enough however, one needs to
   also derive subtle multiplier and commutation  estimates and an appropriate version of the null condition to make everything work.          We shall  summarize the main steps in more detail in the next subsection   after we introduce  some notation and  state the main result.

\subsection{Main results}

To state our main theorem, we first define some necessary notations.
We use the standard polar local coordinate system $(t, r, \om)$ of Minkowski space as well as the null coordinates $u=\frac{t-r}{2}$, $v=\frac{t+r}{2}$. We may use $u_+=1+|u|$. Under the coordinates $(u, v, \om)$ with $\om\in {\mathbb S}^2$, all points with fixed $(u, v)$ form a round sphere with radius $r$, which is denoted as $S{(u, v)}$. We now introduce a null frame $\{L, \Lb, e_1, e_2\}$ with
\[
L=\pa_v=\pa_t+\pa_r,\quad \Lb=\pa_u=\pa_t-\pa_r
\]
and $\{e_1, e_2\}$ an orthonormal basis of the sphere $S(u,v)$. We use $\D$ to denote the covariant derivative associated to the connection field $A$ on the sphere $S(u,v)$, which is defined by $\D_\nu=\Pi^\mu_\nu D_\mu,$ with $\Pi^\mu_\nu=\delta_\nu^\mu+\f12(L^\mu \Lb_\nu+ L_\nu\Lb^\mu).$
For any 2-form $G$, denote the null decomposition under the above null frame by
\begin{equation}\label{12.9.2.17}
\a_B[G]=G_{Le_B},\quad\underline{\a}_B[G]=G_{\Lb e_B},\quad \rho[G]=\f12 G_{\Lb L}, \quad \si[G]=G_{e_1 e_2},\quad B=1, 2.
\end{equation}
We interpret $\a, \ab$ as $1$-forms  tangent to the spheres $S(u,v)$. We refer to such quantities, more generally,  as $S$-tangent  tensors.

Without loss of generality we only prove estimates in the future, i.e., $t\geq 0$. Our analysis  focuses on  the exterior region $\{t+R\leq |x|\}$, for some fixed constant $R\ge 1$. We foliate  this region by the standard  double null foliation  defined be the level surfaces of  the optical functions $u$ and $v$.  Let $\mathcal{H}_{u}$  denote the outgoing null hypersurface $\{t-r=2u\}\cap\{t\geq 0, r\geq R\}$ and $\Hb_v$  the incoming null hypersurface $\{t+r=2v\}\cap\{t\geq 0, r\geq R\}$. We also use $\H_{u}^{v}$ and $\Hb_{v}^{u}$ to denote the truncated hypersurfaces
\begin{align*}
&\H_{u}^{v}:=\{(t, x): \,t-|x|=2u, \quad -2u\leq  t+|x|\leq 2v\},\quad \H_{u}:=\H_{u}^{\infty},\\
&\Hb_{v}^{u}:=\{(t, x):\,t+|x|=2v, \quad -2v\leq t-|x|\leq 2u\}.
\end{align*}
On the initial hypersurface $\{t=0\}$, define $\Sigma_0={\mathbb R}^3\cap \{r\ge R\}$ and
\begin{align*}
\Si_{0}^{u_1, u_2}:=\{(0, x)|-2u_1\leq |x|\leq -2u_2\},\quad \forall\, u_2<u_1\leq -\f12 R,\quad \Si_{0}^{u}=\Si_{0}^{u, -\infty}.
\end{align*}
In the exterior region, let $\mathcal{D}_{u}^{v}$ be the domain
bounded by $\H_{u}^{v}$, $\Hb_{v}^{u}$ and the initial hypersurface:
\[
\mathcal{D}_{u}^{v}:=\{(t, x):\, t-|x|\leq 2u,\quad t+|x|\leq 2v\}.
\]
We fix the convention that the standard volume elements on $\Si_0$,  $\H_u$, $\Hb_v$ and $\mathcal{D}_u^v$ are given by  $r^2 dr d\om$,  $r^2 d v d\om$, $r^2 du d\om$  and $dx dt$ respectively. These will be dropped whenever there is no   possible confusion,  i.e. $\int_{\Si_0} f=\int_{\Si_0} f r^2 dr d\om$,   $\int_{\H_u} f = \int_{\H_u}   f r^2 dv d\om$  etc.

\medskip

We denote by  $E[f, G](\Si)$  the  appropriate  energy-flux of the 2-form $G$ and  complex scalar field $f$ through the hypersurface $\Si$. For the hypersurfaces  of interest to us,
\begin{equation}
\label{generalizedenergy-norms}
\begin{split}
 E[f, G](\Sigma_0)&=\int_{\Sigma_0}(|G|^2+|Df |^2+|f|^2)dx,\quad |G|^2=\rho^2+|\si|^2+\frac{1}{2}(|\a|^2+|\underline{\a}|^2),\\
 E[f, G](\H_u)&=\int_{\H_u}(|D_L f|^2+|\D f|^2+|f|^2+\rho^2+\si^2+|\a|^2) ,\\
 E[f, G](\Hb_v)&=\int_{\Hb_v}(|D_{\Lb}f|^2+|\D f|^2+|f|^2+\rho^2+\si^2+|\ab|^2),
\end{split}
\end{equation}
where $\a, \ab, \rho,\sigma$ are the components of $G$ defined in (\ref{12.9.2.17}), $d\omega$ denotes the standard surface measure on the unit sphere ${\mathbb S}^2$,
and $D$ is the covariant derivative associated to the connection $A$.

On the initial hypersurface, define the chargeless part of the electric field $\tilde{E}$:
\[
\tilde{E}_i=E_i-q_0 r^{-2}\chi_{\{R\leq r\}} \om_i, \quad \mbox{ where } \omega_i=\frac{x_i}{r},
\]
where $q_0$ is the total charge defined in (\ref{defcharge}). Similarly define the chargeless part of the Maxwell field
\[
\tilde{F}=F-q_0 r^{-2}\chi_{\{t+R\leq r\}}dt\wedge dr.
\]

Our assumption on the initial data is that for some positive constant $1<\ga_0 <2$ the following weighted energy
\begin{equation}
\label{IDMKG}
\mathcal{E}_{k, \ga_0}= \sum\limits_{l\leq k}\int_{\Sigma_0}(1+r)^{\ga_0+2l}(|\bar D\bar D^l\phi_0|^2+|\bar D^l\phi_1|^2+|\bar D^l\phi_0|^2+|\bar\nabla^l \tilde{E}|^2+|\bar\nabla^l H|^2)dx
\end{equation}
is small for some positive integer $k\geq  2$, where $\bar D$ is the projection of $D$ to ${\mathbb R}^3\times \{t=0\}$. We denote by  $\nabla$  the Levi-Civita connection  in Minkowski space.  $\bar \nabla$ is the projection of  $\nabla$ to ${\mathbb R}^3\times \{t=0\}$. Note that the total charge is defined in terms of an integral on ${\mathbb R}^3$, it is not bounded by the energy $\E_{0,\ga_0}$ which is only defined for ${\mathbb R}^3\cap \{r\ge R\}$. Therefore it can be large even if $\E_{0,\ga_0}$ is small.

 We are ready to  state the main theorem of this paper.
\begin{thm}[Main theorem]\label{m.thm}
Consider the Cauchy problem for (\ref{EQMKG}) with the admissible initial data set $(\phi_0, \phi_1, E, H)$. There exists a positive constant $\ep_0$, depending only on $1<\ga_0<2$, $|q_0|$ and $\ep$ such that if $\E_{2,\ga_0}<\ep_0$,  the unique local solution  $(F, \phi)$ of (\ref{EQMKG}) can be globally extended \begin{footnote}{We refer the reader to \cite{Moncrief1, Moncrief2}  for the standard  global existence without  asymptotic behavior.}\end{footnote} in time on the exterior region $\{(t, x):\, t+R\leq |x|\}$, with a  fixed constant  $R\ge 1$.

\begin{itemize}
\item[(1)] The global solution verifies the following pointwise estimates,
\begin{align*}
r^2|\D\phi|^2+u_+^2|D_{\Lb}\phi|^2+r^2|D_L \phi|^2 &\leq C \E_{2,\ga_0} r^{-\frac{5}{2}+\ep}u_+^{\f12-\ga_0},\quad  |\phi|^2 \leq C \E_{2,\ga_0} r^{-3}u_+^{-\ga_0};\\
|\tilde{\rho}|^2+|\a|^2+|\si|^2&\leq C \E_{2,\ga_0} r^{-2-\ga_0}u_+^{-1},\quad  |\ab|^2\leq C \E_{2,\ga_0} r^{-2}u_+^{-\ga_0-1},
\end{align*}
where $\ep>0$ is any positive constant. Here $\tilde \rho=\rho[\tilde F]$ and the other curvature components are for the full Maxwell field $F$.
\item[(2)] The  following generalized energy estimates (see the notation in  \eqref{generalizedenergy-norms}) hold true
\begin{align*}
&E[D_Z^k\phi, \cL_Z^k\tilde{F}](\H_{u_1}^{-u_2})+E[D_Z^k\phi, \cL_Z^k\tilde{F}](\Hb_{-u_2}^{u_1})\le C (u_1)_+^{-\ga_0+2\zeta(Z^k)}\E_{2,\ga_0} ,\\
&\int_{\H_{u_1}^{-u_2}}r|D_L D_Z^k\phi|^2 +\int_{\Hb_{-u_2}^{u_1} }r(|\D D_Z^k\phi|^2+|D_Z^k\phi|^2) \le C(u_1)_+^{1-\ga_0+2\zeta(Z^k)}\E_{2,\ga_0},\\
&\int_{\H_{u_1}^{-u_2}}r^{\ga_0}|\a[\cL_Z^k \tilde{F}]|^2 +\iint_{\mathcal{D}_{u_1}^{-u_2}}r^{\ga_0-1}|(\a, \rho,\sigma)[\cL_Z^k\tilde{F}]|^2 \\
&\qquad \qquad+\int_{\Hb_{-u_2}^{u_1} }r^{\ga_0}(|\rho[\cL_Z^k\tilde{F}]|^2+|\si[\cL_Z^k\tilde{F}]|^2)  \le C (u_1)_+^{2\zeta(Z^k)}\E_{2,\ga_0}
\end{align*}
for all $u_2<u_1\leq -\frac{R}{2}$, $Z^k=Z_1 Z_2\ldots Z_k$ with $k\le 2$ and $Z_i\in \Gamma$, where $\Gamma$ is the set of  generators
  of the Poincar\' e group.\begin{footnote}{Here $D_Z^l=D_{Z^l}:=D_{Z_1} \ldots D_{Z_l},\,  \Lie_Z^l=\Lie_{Z^l}:=\Lie_{Z_1}\ldots \Lie_{Z_l}$, and $D_{Z^0}, \Lie_{Z^0}$ are both the identity map.} \end{footnote}The definition of the signature function $\zeta(\cdot)$ can be found in Section \ref{1.9.1.18}.
\end{itemize}
 The constant $C$ in (1) and (2) depends only on $\ga_0$, $|q_0|$ and $\ep$.
\end{thm}

\begin{remark}
If $\ga_0\geq 2$, pointwise decay estimates as well as weighted energy estimates can be improved according to a  slightly modified argument. Moreover if the initial data belong to higher order weighted Sobolev space ($k>2$), we then can also derive the associated higher order weighted energy estimates. Note that the theorem can be adapted to large data, provided that $R$ is sufficiently large. Indeed, if the total energy in ${\mathbb R}^3$ is bounded, the energy $\E_{2,\ga_0}$ defined on $\{r\ge R\}$ can be sufficiently small.
\end{remark}
\begin{remark}
The  estimates   mentioned above   can be used as
boundary conditions on the outgoing null boundary  of the causal  future of a sufficiently large compact set
 where one can apply   the hyperboloidal  approach  of   \cite{Maria:mMKG:small}  to   derive  global asymptotic   results  for general initial conditions. \end{remark}

\begin{remark}
Our approach can be applied to (mMKG) equations in Minkowski spacetimes of other dimensions. It can also be used to remove the restriction of the compactly supported Cauchy data for the  $2$-D problems treated by the hyperboloid foliation such as  \cite{Willie}. More importantly our  approach can also be used to study quasilinear wave and Klein-Gordon systems, such as the model problem in \cite{Ionescu:waveKG} as well as  the Einstein-Klein-Gordon equations\footnote{ The last two authors of this paper  are  in  the process of  completing   this  goal.}.
\end{remark}

\begin{remark}
As mentioned earlier  the exterior region provides considerable flexibility  in  allowing   high weights in  $u_+$ for the various flux integrals used here, as  long as  the initial  data  has sufficient decay in $r$. An obvious restriction however is given by the charge which limits the  decay  in $r$ for the  $\rho[F]$ component of the data. The way to overcome such restriction is to define the energy of the data for $\tilde F$, the chargeless part of $F$.
\end{remark}

We now briefly summarize our method.  As expected,  if  we   apply the standard   multiplier approach to the linear Klein-Gordon equation with the multiplier $rL$, the energy identity generates  the bulk term
\[
B:=\int_{\{r\ge R+t \}}|\phi|^2 dx dt
\]
  with  an unfavorable  sign.
 Note that  the flux of  the standard energy estimate provides a  bound  for $ \int_{\H_u}|\phi|^2 $ in terms of the initial energy.  One can improve this, in the exterior region, to  a bound of
 $u_+^{\ga} \int_{\H_u}|\phi|^2$  in terms of  the corresponding weighted  energy norm of the data. Thus, by  integrating $ \int_{\H_u}|\phi|^2  $   with respect to $u$ (for $\ga>1$),   we can control the  bulk  term $B$,  as long as the  initial weighted  energy is bounded.
  This  multiplier approach,  combined  with      commutation with   $\Om_{\a\b}$ and standard translations, $T_\a=\pr_\a$,   is  thus  sufficient to derive the desired pointwise decay estimates for solutions of linear   Klein-Gordon equations in the exterior  and can also be adapted to   control solutions to nonlinear  KG equations.

  \medskip

  Now  consider the  mMKG system.  One can show, as  before,   that the total  energy flux through the outgoing null hypersurface $\H_u$ in the exterior region   decays, sufficiently fast,  with respect to  $u$, i.e.
 \[
 \int_{\H_u}(|D_L\phi|^2+|\D\phi|^2+|\phi|^2 )\les u_+^{-\ga_0}\E_{0, \ga_0}, \qquad \ga_0>1,
 \]
 provided that the  initial data is  bounded in the corresponding  weighted energy space $\E_{0, \ga_0}$. The key observation, once more, is that this energy flux decay is sufficient to bound the only unfavorable term
 \[
 \int_{u\leq u_0}(\int_{\H_u}|\phi|^2) du\les {u_0}_+^{1-\ga_0}\E_{0, \ga_0},
 \]
 generated  when using the vector field $rL$ as multiplier for the coupled system\footnote{ The bad term is due, of course, to the  Klein-Gordon  component of the system.}.

 \medskip

 Things become  a lot   more complicated  when we try to derive  the higher order derivative estimates due to the complexity of the  quadratic  error terms generated  in the process. At the top level, when we commute with two vector fields $X, Y$ the main error terms are due    to the commutator\footnote{ See   (\ref{12.7.6.17}) for a  precise form of  $Com$.} $Com:=[D_X D_Y, \Box-1]\phi$.
  We make  use of  the technique of  double commutator in  \cite{yangMKG} to decompose $Com$   into a  combination of trilinear forms, such as
$Q(F, D_X \phi, Y)$  and $Q(\Lie_X F, \phi, Y)$, where the definition of $Q$ can be found in (\ref{12.9.1.17}). To control  $Com$  we  need to  take into account
that   $Q$ verifies the null condition with respect to the fields\footnote{This means roughly that the radiative components of  $F$, i.e. $\underline{\a}[F]$,  does not interact with $D_{\Lb}\phi$.}  $(\phi, F)$. Consider the case  when $Y=\Omega_{0i}$ is a  boost, as a typical   example.
In that case we can estimate, schematically, (see Lemma \ref{lem:Est4commu:1} for the complete inequalities),
\begin{align}\label{1.2.1.18}
|Q(F, D_X\phi, Y)|&\les r |F\cdot D_L D_X \phi|+\cdots;\quad  |Q(\Lie_X F,\phi,Y)|\les r|\Lie_X F\cdot D_L \phi|+\cdots.
\end{align}

To bound the standard energy flux for $E[D_X D_Y \phi](\H_u)+E[D_X D_Y \phi](\Hb_v)$ (which can be found in the first inequality in Theorem \ref{m.thm} (2)),  we need to control \begin{footnote} {For any one-form $V$, we fix the convention that $V_0=V(\pa_t)$. Lifting it by Minkowski metric gives $V^0$.}
 \end{footnote}
\begin{equation}\label{1.2.3.18}
 \int_{\{r>t+R\}}\big(|Q(F, D_X \phi, Y)|+|Q(\Lie_{X} F, \phi, Y)|\big)|D_0 D_X D_Y \phi| dx dt =I_1+I_2.
\end{equation}
     Consider the first term $I_1$ in (\ref{1.2.3.18}) in view of (\ref{1.2.1.18}) and decomposing $2D_0=D_L +D_{\Lb}$, we need to bound
\begin{equation*}
I_1 \les \int_{\{r>t+R\}} r|F\cdot D_L D_X \phi|(|D_{\Lb} D_X D_Y \phi|+|D_L D_X D_Y\phi|) dx dt.
\end{equation*}
For simplicity, we only discuss the treatment of the chargeless part of $F$, for which we have
\begin{align*}
I_1   \les \|r^{-\f12-\ep}(|D_{\Lb} D_X D_Y \phi|+|D_L D_X D_Y|)\|_{L^2(\{r>t+R\})}\|r^{\frac{3}{2}+\ep} \tilde F \cdot D_L D_X\phi\|_{L^2\{r>t+R\}}+\cdots
\end{align*}
 The first factor can be bounded by the energy fluxes on  $\Hb_v$ and $\H_u$, followed with  direct integration. To bound the second factor we note that,  since
 the  radiative $\ab$ component of $\tilde F$ decays only  like  $ r^{-1}$ in $r$, in view of  its expected decay in Theorem \ref{m.thm} (1),  we write,
 \begin{align*}
 \int_{\{r>t+R\}} r^{3+2\ep}|\tilde F|^2 | D_L D_X\phi |^2&\les   \E_{2,\ga_0}\int_{\{r>t+R\}}r^{-1+2\ep} u_+^{-\ga_0-1}| r D_L D_X\phi |^2 \\
 &\les\E_{2,\ga_0} \sup_{u \le  -\frac 1 2 R}  \|rD_L D_X\phi\|_{L^2(\H_u)} ^2,
 \end{align*}
 which  requires us to control
 \begin{equation*}
 %\label{improvedboundforDlDXphi}
     \|rD_L D_X\phi\|_{L^2(\H_u)}.
     \end{equation*}
  We can repeat the above estimate  in the region $\{r\ge t-2u_0\}$ with $u_0\le -\f12 R$ so as to keep track of the negative power of $u_0^+$ in the bounds.
To  derive the improved  bound for   $\|rD_L D_X\phi\|_{L^2(\H_u)}$,  we note the  simple algebraic identity, $$v|L f|\les \sum_i|\Omega_{0i}f|+\sum_{1\le i<j\le 3}|\Omega_{ij}f|+|u| |\pa f| $$  for any smooth  function $f$. By using the standard energy flux on $\H_u$ and  $r\le 2v$ we have
\begin{equation}\label{1.2.4.18}
\|r D_L D_X \phi\|^2_{L^2(\H_u)}\les \sum_{\mu, \nu}\|D_{\Omega_{\mu\nu}}D_X \phi\|^2_{L^2(\H_u)}+u_+^2 \|D D_X\phi\|^2_{L^2(\H_u)}.
\end{equation}
 In view of  the top-order standard energy flux bound, we expect that $$\|D D_X \phi\|^2_{L^2(\H_u)}\les u_+^{-\ga_0-2} \E_{2,\ga_0}. $$ Thus,  estimating  the other terms on the right of (\ref{1.2.4.18}) by $u_+^{-\ga_0}\E_{2, \ga_0}$,  we  obtain the desired estimate
$$
\|r D_L D_X \phi\|^2_{L^2(\H_u)}\les u_+^{-\ga_0}\E_{2,\ga_0}.
$$

To treat the term $I_2$ in (\ref{1.2.3.18}), we need  the pointwise control for $D\phi$.  It is worthwhile to point out that the weighted energy control (see the second inequality in Theorem \ref{m.thm} (2)) allows us to obtain a set of strong pointwise decay for the scalar field. However to treat  the leading term generated  by  (\ref{1.2.1.18}), we need a further improvement for the pointwise decay  of  $D_L \phi$. This is obtained by making use of  the same algebraic identity as above   combined with   a  pointwise estimate\footnote{Here  $Z$ is a  generator
  of the Poincar\' e group.} for $D_Z \phi$, with the help of the top order weighted  energy for  $\phi$  and global Sobolev inequalities.

To control the top order weighted estimate for the scalar field
we proceed in the same manner  except that we need to make use  of
the stronger decay  in $u_+$ in the treatment of the error terms, see
Proposition \ref{1.27.3.18}.  To derive the  weighted energy estimates\footnote{The standard energy estimates are,  of course,  much simpler. } for the  Maxwell field  we use again the multiplier method   based  on   the vector fields  $r^p L$. In this case    however we can  choose  $1\leq p\leq \ga_0<2$. The error terms  generated  in this case  are   due   to  the inhomogeneous terms  of the Maxwell equations  which  depend  quadratically   on  the scalar field and its derivatives. This is, roughly,  the reason why  we can get stronger  r-weighted estimates on the Maxwell field $F$  than on the KG field $\phi$.  Our main theorem then follows by using a standard bootstrap argument.

At last, we remark that our result does not require small charges. This is achieved by carefully separating the terms related to the charge in the analysis.  We direct the  readers   interested   in this aspect of our result to
 the proof of Lemma \ref{prop:Q:need:r:all} and Proposition \ref{1.27.3.18}.

 \medskip

\textbf{Acknowledgments.}  The first author    is partially supported by  the NSF grant   DMS 1362872.       The second author would like to thank the Zilkha Trustees for offering partial travel support. The third author is  partially supported by NSFC-11701017.

\section{Preliminaries and energy identities}
\label{notation}

The proof of the main theorem is based on an energy method.  We will mainly analyze the energy fluxes through various kinds of null hypersurfaces.
To derive the pointwise bound of the solution, we need a  global Sobolev embedding inequality adapted to the null hypersurfaces.  In this section, we establish a preliminary Sobolev inequality and energy identities. In order to control higher order energy fluxes, we need to commute various differential operators with the operator  $\Box_A-1$. Such commutators are treated in the last subsection.

Without loss of generality, we may assume the positive constant $\ep$ in the Main Theorem verifies $0<\ep\le \frac{1}{10}(\ga_0-1)$. We make a convention in the sequel that the implicit constant in $A\les B$ depends only on $\ep$, $\ga_0$ and $|q_0|$.
\subsection{Sobolev inequalities}
We may write the integral of a real scalar function $f$ on any surface $S$ as $\int_{S}f$, where the volume element on the surface $S$ is omitted  for simplicity.
We have the following Sobolev inequality.
\begin{lem}
\label{lem:globalSobolev}
\begin{itemize}
\item[(1)]
For any smooth function $f$ and constants verifying the relation $2\ga=\ga_0'+2\ga_2$, we have, for all $(u, v),  \,  -u<v\le v_*$,
\begin{equation}
\label{eq:globalSobolev:massless}
\begin{split}
\sup\limits_{S{(u, v)}}|r^{\ga} f|^4
&\les \sum\limits_{l\leq 1, 1\le i<j\le 3} \int_{S(u, -u)}|r^\ga \Om_{ij}^l f|^4 r^{-2}+\sum\limits_{k\leq 2,1\le i<j\le 3}\int_{\H_u^{-u, v_*}}r^{2\ga_2} |\Om_{ij}^k f|^2 r^{-2}\\
&\times \sum\limits_{l\leq 1, 1\le i<j\le 3 }\int_{\H_u^{-u, v_*}} r^{\ga_0'}|L \Om^l_{ij}(r^{\ga} f)|^2 r^{-2}.
\end{split}
\end{equation}
\item[(2)]The same estimate holds true for any smooth complex scalar field $\phi$ with covariant derivative $D$ associated to the connection field $A$.
\item[(3)]
 For any $S$-tangent  tangent tensor field $H$, we have, for all $(u, v),  \,  -u<v\le v_*$,
\begin{align}
\notag
\sup\limits_{S{(u, v)}}|r^{\ga} H|^4
& \les \sum\limits_{l\leq 1, 1\le i<j\le 3} \int_{S(u, -u)}|r^\ga {\sl\cL}_{\Om_{ij}}^l H|^4 r^{-2}+\sum\limits_{k\leq 2,1\le i<j\le 3}\int_{\H_u^{-u, v_*}}r^{2\ga_2} |{\sl\cL}^k_{\Om_{ij}} H|^2 r^{-2}\\
\label{12.25.2.17}
&\cdot\sum\limits_{l\leq 1, 1\le i<j\le 3 }\int_{\H_u^{-u, v_*}} r^{\ga_0'}|{\sl\nab}_L {\sl\cL}_{\Om_{ij}}^l (r^{\ga} H)|^2 r^{-2},
\end{align}
where, for any vector field $Z$,  ${\sl\cL}_Z H$ denotes the projection of $\cL_Z H$ to the sphere $S(u,v)$. Here $\sl{\nabla}_L H$ denotes the projection of $\nabla_L H$ to $S(u,v)$.
\end{itemize}
\end{lem}
\begin{proof}
The case for $\ga=\ga_2=\frac{3}{2}$ was first established in \cite{KN}. The proof for this general version of this lemma is a minor modification of the original one. For reader's benefit, we give the proof here.  The crucial idea is based on the following Poincar\'e inequality
\begin{align*}
\int_{S}|\Phi-\bar \Phi|^2 \leq C\left(\int_{S}|\nabb \Phi|\right)^2
\end{align*}
for any smooth function $\Phi$ and $2$-sphere $S$, where the constant $C$ depends only on the sphere and $\bar \Phi$ is the mean value of the function $\Phi$ on the surface. For the case when the sphere $S$ is the standard unit sphere ${\mathbb S}^2$,  with $C$ be the universal constant, we  can apply the above inequality to $|f|^3$ to derive that
\begin{align*}
\int_{{\mathbb S}^2}|f|^6 d\om \leq C \int_{{\mathbb S}^2}|\pa_\om f|^2 d\om \cdot \int_{{\mathbb S}^2} |f|^4 d\om+C\left(\int_{{\mathbb S}^2} |f|^3\right)^2\leq C \int_{{\mathbb S}^2}\left(|f|^2+|\pa_\om f|^2\right) d\om \cdot \int_{{\mathbb S}^2} |f|^4 d\om.
\end{align*}

For any constants $\ga$, $\ga_1$, $\ga_2$ satisfying $3\ga=\ga_1+2\ga_2$, we in particular have
\begin{align*}
\int_{{\mathbb S}^2}|r^\ga f|^6 d\om \leq C \int_{{\mathbb S}^2}\left(|r^{\ga_1} f|^2+|\pa_\om (r^{\ga_1} f)|^2\right) d\om \cdot \int_{{\mathbb S}^2} |r^{\ga_2} f|^4 d\om.
\end{align*}
 Let $\ga$, $\ga_0'$, $\ga_1$ and $\ga_2$ verify the following relations
\[
\ga_0'+6\ga_1=6\ga,\quad 3\ga_1=2\ga+\ga_2.
\]
By integrating along the outgoing null hypersurface $\H_u$, we can obtain
\begin{align*}
&\int_{S(u, v_1)}|r^{\ga} f|^4 r^{-2}  \leq \int_{S(u, -u)}|r^\ga f|^4 r^{-2} +4\int_{\H_u^{-u, v_1}}|L(r^\ga f)| |r^\ga f|^3 dvd\om \\
\leq &\int_{S(u, -u)}|r^\ga f|^4 r^{-2} +4\left(\int_{\H_u^{-u, v_1}}r^{\ga_0'}|L(r^\ga f)|^2 dvd\om \right)^\f12 \left(\int_{\H_u^{-u, v_1}}|r^{\ga_1} f|^6 dvd\om\right)^\f12\\
\leq& \int_{S(u, -u)}|r^\ga f|^4 r^{-2}  +4C \left(\sup\limits_{ v} \int_{S(u, v)}|r^{\ga} f|^4 r^{-2}\right)^\f12\\
 &\qquad\qquad\cdot\left(\int_{\H_u^{-u, v_1}}r^{\ga_0'}|L(r^\ga f)|^2 dvd\om \right)^\f12  \left(\int_{\H_u^{-u, v_1}}(|r^{\ga_2} f|^2+|\pa_\om (r^{\ga_2} f)|^2) dvd\om\right)^\f12.
\end{align*}
Taking supremum on the left hand side with respect to $v_1\in(-u, v_*]$ gives
\begin{equation}\label{12.25.1.17}
\begin{split}
\sup_{-u<v\le v_*}\int_{S(u, v)}|r^\ga f|^4 r^{-2}
&\les \int_{S(u, -u)}|r^\ga f|^4 r^{-2} + \int_{\H_u^{-u, v_*}}r^{\ga_0'}|L(r^\ga f)|^2 dvd\om \\
&\qquad \qquad \qquad \qquad \cdot \int_{\H_u^{-u, v_*}}(|r^{\ga_2} f|^2+|\pa_\om (r^{\ga_2} f)|^2) dvd\om.
\end{split}
\end{equation}
Finally, by using
 the standard   Sobolev embedding on the sphere $S(u,v)$, we have
\begin{equation*}
\sup_{S(u,v)}|\varphi|\les \sum_{l\le 1, 1\le i<j\le 3}(\int_{S(u,v)} |\Omega^l_{ij}\varphi|^4 r^{-2})^\frac{1}{4}.
\end{equation*}
 The desired estimate \eqref{eq:globalSobolev:massless} then follows from  (\ref{12.25.1.17}). (2) and (3) of  Lemma \ref{lem:globalSobolev} can be  proved in the same way.
\end{proof}

%In application, for massless waves, the lemma will be used with $\ga=\ga_2$ while for massive waves $\ga_2=1$.

\subsection{Energy identities}

In this subsection, we derive fundamental energy identities for the massive MKG equations.

For any 2-form $\G$, satisfying the Bianchi identity \eqref{bianchi},   any complex scalar field $\phi$ and connection field $A$, we define the associated energy momentum tensor
\begin{equation*}
\begin{split}
 T[\phi, \mathcal{G}]_{\a\b}&=\G_{\a\mu}\G_\b^{\;\mu}-\frac{1}{4}m_{\a\b}\G_{\mu\nu}\G^{\mu\nu}+\Re\left(\overline{D_\a\phi}D_\b\phi\right)-\f12 m_{\a\b}(\overline{D^\mu\phi}D_\mu\phi+|\phi|^2),
\end{split}
\end{equation*}
where $m_{\a\b}$ is the Minkowski metric and $D_\a$ denotes  the covariant derivative  associated to   the connection field  $A$.  Given a vector field $X$, we have the following identity
\[
\pa^\mu(T[\phi,\G]_{\mu\nu}X^\nu) = \Re((\Box_A-1) \phi X^\nu \overline{D_\nu\phi})+X^\nu F_{\nu\mu}J^\mu[\phi]+\pa^\mu \G_{\mu\b}\G_{\nu}^{\;\b}X^{\nu}+T[\phi,\G]^{\mu\nu}\pi^X_{\mu\nu},
\]
where $\pi_{\mu\nu}^X=\f12 \mathcal{L}_X m_{\mu\nu}$ is the deformation tensor of the vector field $X$ in Minkowski space and $J_\mu[\phi]=\Im(\phi\cdot \overline{D_\mu\phi})$. We also note that the term  $F=F[A]$ appears from commuting covariant derivatives of $\phi$ as in (\ref{1.17.1.18}).
Throughout this paper, we raise and lower indices with respect to the Minkowski metric $m_{\mu\nu}$.

Take any smooth function $\chi$. We have the following equality
\begin{align*}
 \f12\pa^{\mu}\left(\chi \pa_\mu|\phi|^2-\pa_\mu\chi|\phi|^2\right)= \chi (\overline{D_\mu\phi}D^\mu\phi+|\phi|^2) -\f12\Box\chi\cdot|\phi|^2+\chi \Re((\Box_A-1)\phi\cdot \overline\phi).
\end{align*}
Let $X,Y$ be smooth vector fields. We now define the vector field $\widetilde{P}^{X,Y}[\phi, \G]$ with components
\begin{equation}
\label{mcurent} \widetilde{P}^{X,Y}_\mu[\phi,\G]=T[\phi,\G]_{\mu\nu}X^\nu -
\f12\pa_{\mu}\chi \cdot|\phi|^2 + \f12 \chi\pa_{\mu}|\phi|^2+Y_\mu,
\end{equation}
where the vector field $Y$ may depend on the scalar field $\phi$. We then have the equality
\begin{align*}
\pa^\mu \widetilde{P}^{X,Y}_\mu[\phi,\G] =&\Re((\Box_A-1) \phi(\overline{D_X\phi}+\chi\overline \phi))+div(Y)+X^\nu F_{\nu\mu}J[\phi]^\mu+\pa^\mu \G_{\mu\ga}\G_{\nu}^{\;\ga}X^{\nu}
\\&+T[\phi,\G]^{\mu\nu}\pi^X_{\mu\nu}+
\chi (\overline{D_\mu\phi}D^\mu\phi+|\phi|^2) -\f12\Box\chi\cdot|\phi|^2.
\end{align*}
Here the operator $\Box$ is the wave operator in Minkowski space and the divergence of the vector field $Y$ is also taken in the Minkowski space and $J[\phi]_\mu=\Im(\phi\cdot \overline{D_\mu \phi})$.

Now take any region $\mathcal{D}$ in $\mathbb{R}^{3+1}$. Assume on this region the scalar field $\phi$ and the 2-form $\G$ satisfies the following linear equations
\begin{equation}
\label{eq:mMKG:lin}
 \pa^\nu \G_{\mu\nu}=J_{\mu},\quad \Box_A\phi-\phi=h.
\end{equation}
Here we note that the covariant operator $\Box_A$ is associated to the 2-form $F$. Then using the Stokes' formula, the above calculation leads to the following energy identity
\begin{align}
\notag &\iint_{\mathcal{D}} F_{X\mu} J[\phi]^\mu- \G_{X\ga}J^\ga+\Re(h\cdot (\overline{D_X\phi}+\chi \bar \phi)) d\vol\\
\notag
&+ \iint_{\mathcal{D}}div(Y)+T[\phi,\G]^{\mu\nu}\pi^X_{\mu\nu}+
\chi (\overline{D_\mu\phi}D^\mu\phi+|\phi|^2 )-\f12\Box\chi\cdot|\phi|^2d\vol\\
&=\iint_{\mathcal{D}}\pa^\mu \widetilde{P}^{X,Y}_\mu[\phi,\G]d\vol=\int_{\pa \mathcal{D}}i_{\widetilde{P}^{X,Y}[\phi,\G]}d\vol,
\label{energyeq}
\end{align}
where $\pa\mathcal{D}$ denotes the boundary of the domain $\mathcal{D}$ and $i_Z d\vol$ denotes the contraction of the volume form $d\vol$
with the vector field $Z$ which gives the surface measure of the
boundary. For example, for any basis $\{e_1, e_2,
\ldots, e_n\}$, we have $i_{e_1}( de_1\wedge de_2\wedge\ldots
de_k)=de_2\wedge de_3\wedge\ldots\wedge de_k$.

In this paper, the domain $\mathcal{D}$ will be a regular region bounded by a level set of  $t$, an outgoing null hypersurfaces $\H_u$ and an incoming null
hypersurfaces $\Hb_v$. We now compute $i_{\widetilde{P}^{X,Y}[\phi,\G]}d\vol$ on these three kinds of hypersurfaces.
%On the level set of $t$, the surface measure is a function times $dx$.
Recall the volume form in Minkowski space
 \[
d\vol=dx\wedge dt=-dt\wedge dx.
\]
%Here note that $dx$ is a $3$-form.
 We thus can show that on $\Sigma_0$
\begin{align}
\notag
 i_{\widetilde{P}^{X,Y}[\phi,\G]}d\vol
 =&-(\Re(\overline{D^0\phi}
D_X\phi)-\f12 X^0\overline{D^\ga\phi}D_\ga\phi-\f12 X^0 |\phi|^2-\f12 \pa^0\chi \cdot
|\phi|^2+\f12\chi\pa^0|\phi|^2+Y^0\\
\label{curlessR}
&+\G^{0\mu}\G_{\nu\mu}X^\nu-\frac{1}{4}X^0\G_{\mu\nu}\G^{\mu\nu})dx.
\end{align}
On the outgoing null hypersurface $\H_u$, we can write the volume form
\[
d\vol=r^2dr\wedge dt\wedge d\om=2r^2dv\wedge du\wedge d\om=-2r^2 du\wedge dv\wedge d\om.
\]
%Here $u=\frac{t-r}{2}$, $v=\frac{t+r}{2}$ are the null coordinates and $d\om$ is the standard surface measure on the unit sphere.
Notice that $\Lb=\pa_u$. We can compute on $\H_u$ that
\begin{align}
\notag
i_{\widetilde{P}^{X,Y}[\phi, \G]}d\vol=&-2(\Re(\overline{D^{\Lb}\phi}
D_X\phi)-\f12 X^{\Lb}\overline{D^\ga\phi}D_\ga\phi-\f12X^{\Lb}|\phi|^2-\f12
\pa^{\Lb}\chi|\phi|^2+\f12\chi \pa^{\Lb}|\phi|^2+Y^{\Lb}\\
\label{curStau}
&+\G^{\Lb\mu}\G_{\nu\mu}X^\nu-\frac{1}{4}X^{\Lb}\G_{\mu\nu}\G^{\mu\nu})r^2dv\wedge d\om.
\end{align}
Similarly, on $\Hb_v$ we
have
\begin{equation}
\label{curnullinfy}
\begin{split}
  i_{\widetilde{P}^{X,Y}[\phi,\G]}d\vol=&2(\Re(\overline{D^{L}\phi}
D_X\phi)-\f12 X^{L}\overline{D^\ga\phi}D_\ga\phi-\f12 X^L|\phi|^2-\f12
\pa^{L}\chi|\phi|^2+\f12\chi \pa^{L}|\phi|^2+Y^{L}\\
&+\G^{L\mu}\G_{\nu\mu}X^\nu-\frac{1}{4}X^L\G_{\mu\nu}\G^{\mu\nu})r^2du\wedge d\om.
\end{split}
\end{equation}
%We remark here that the above formulae hold for any vector fields $X$, $Y$ and any function $\chi$.

We now establish the $r$-weighted energy identities for solutions of the linear massive MKG equations \eqref{eq:mMKG:lin} in the exterior region for the domain $\mathcal{D}_{u_1}^{ v_1}$. In the energy identity \eqref{energyeq}, we choose the vector fields $X$, $Y$ and the function $\chi$ as follows
\[
 X=r^{p}L, \quad Y=\frac{p}{2}r^{p-2}|\phi|^2L,\quad  \chi=r^{p-1}.
\]
  We  then compute the non-vanishing components of the deformation tensor for the vector field $X$:
\[
\pi_{L\Lb}^X=-pr^{p-1},\quad \delta^{AB}\pi^X_{e_Ae_B}=r^{p-1}, \, \, A,B=1,2.
\]
In the sequel, $\a, \rho, \sigma$ are the components of $\G$ defined in (\ref{12.9.2.17}).
Therefore we can show that
\begin{align*}
 & div(Y)+T[\phi, \G]^{\mu\nu}\pi_{\mu\nu}^X+\chi (\overline{D^{\mu}\phi}D_\mu\phi+|\phi|^2)-\f12\Box\chi |\phi|^2\\
&=\frac{p}{2}r^{-2}L(r^{p}|\phi|^2)+\f12r^{p-1}\left(p(|D_L\phi|^2+|\a|^2)+(2-p)(|\D\phi|^2+\rho^2+\si^2)\right)\\
&\quad -\f12 p(p-1)r^{p-3}|\phi|^2-\f12 p r^{p-1}|\phi|^2\\
&=\f12 r^{p-1}\left(p(r^{-2}|D_L(r\phi)|^2+|\a|^2)+(2-p)(|\D\phi|^2+\rho^2+\si^2)\right)-\f12 p r^{p-1}|\phi|^2.
\end{align*}
Next we compute the boundary terms.
\begin{align*}
 \int_{\H_{u}}i_{\widetilde{P}^{X,Y}[\phi, \G]}d\vol&=\int_{\H_{u}}\{r^{p}(|D_L(r\phi)|^2+r^2|\a|^2)-\f12 L(r^{p+1}|\phi|)\}r^{-2},\\
\int_{\Hb_v}i_{\widetilde{P}^{X,Y}[\phi, \G]}d\vol&=-\int_{\Hb_v}\{r^{p+2}(|\D\phi|^2+|\rho|^2+|\si|^2+|\phi|^2)+\f12\Lb(r^{p+1}|\phi|^2)\}r^{-2},\\
\int_{\Si_0^{u_1,u_2}}i_{\widetilde{P}^{X,Y}[\phi,\G]}d\vol&=\f12\int_{\Si_0^{u_1,u_2}}r^{p+2}\big(r^{-2}|D_L(r\phi)|^2+|\D\phi|^2+|\phi|^2+|\a|^2+|\rho|^2+\si^2\big)\\
&\quad\quad-\pa_r(r^{p+1}|\phi|^2) d\om dr.
%\int_{r=R, \tau_1\leq t\leq \tau_2 }i_{\widetilde{P}^{X,Y}[\phi, \tilde{F}]}d\vol&=\f12\int_{\tau_1}^{\tau_2}\int_{\om}r^p(|D_L\psi|^2-|\D\psi|^2-|\psi|^2+r^2(|\tilde \a|^2-|\tilde\rho|^2-|\tilde\si|^2))\\
%&\quad\quad-\pa_t(r^{p+1}|\phi|^2) d\om dt.
\end{align*}
For the domain $\mathcal{D}_{u_1}^{-u_2}$ in the exterior region we have the identity
\begin{align*}
 &\int_{\H_{u_1}^{-u_2}}L(r^{p+1}|\phi|^2)dv d\om-\int_{\Hb_{-u_2}^{u_1}}\Lb(r^{p+1}|\phi|^2)du d\om-\int_{\Si_0^{u_1,u_2}}\pa_r(r^{p+1}|\phi|^2)d\om dr=0.
\end{align*}
The above calculations then lead to the first energy identity in the following lemma.
\begin{lem}
Assume that  the triplet  $(\mathcal{G}, \phi, A)$ verifies \eqref{eq:mMKG:lin}, i.e.,
\begin{equation*}
 \pa^\nu \G_{\mu\nu}=J_{\mu},\quad \Box_A\phi-\phi=h.
\end{equation*}
Then the following identities hold true in the exterior region $\{r\ge R+t\}$  (with $F=dA$).
\begin{itemize}
\item[(1)]
For all $u_2<u_1 \leq -\frac{R}{2}$ with $X=r^p L$, $0<p<2$,   $\a, \rho, \sigma$   the null components of $\G$  (as defined in (\ref{12.9.2.17})),
\begin{align}
\notag
&\iint_{\mathcal{D}_{u_1}^{-u_2}} F_{X\mu} J[\phi]^\mu- \G_{X\ga}J^\ga+\Re(h\cdot (\overline{D_X\phi}+\chi \bar \phi)) \\
\notag
  &+\f12\iint_{\mathcal{D}_{u_1}^{- u_2}}r^{p-1}\Big(p(r^{-2}|D_L(r\phi)|^2+|\a|^2)+(2-p)(|\D\phi|^2+|\rho|^2+\si^2)\Big)\\
\label{pWescaout}
&+\int_{\H_{u_1}^{-u_2}}r^p(r^{-2}|D_L(r\phi)|^2+|\a|^2)+\int_{\Hb_{-u_2}^{u_1} }r^{p}(|\D\phi|^2+|\rho|^2+\si^2+|\phi|^2)\\
\notag
=&\f12\int_{\Si_0^{u_1,u_2}}r^p(r^{-2}|D_L(r\phi)|^2+|\D\phi|^2+|\phi|^2+|\a|^2+|\rho|^2+\si^2)+\f12\iint_{\mathcal{D}_{u_1}^{-u_2}}pr^{p-1}|\phi|^2.
\end{align}

\item[(2)]  For all $u_2<u_1\leq -\frac{R}{2}$ the following version of  energy  identity holds true
\begin{equation}
 \label{eq:EEst:id:out}
 \begin{split}
&2\iint_{\mathcal{D}_{u_1}^{-u_2}} \left(F_{0\mu} J[\phi]^\mu- \G_{0\ga}J^\ga+\Re(h\cdot \overline{D_0\phi})\right)\\
&+E[\phi, \G](\H_{u_1}^{-u_2})+E[\phi, \G](\Hb_{-u_2}^{u_1})= E[\phi, \G](\Si_0^{u_1, u_2}).
\end{split}
\end{equation}

\end{itemize}
\end{lem}
Indeed, to see (\ref{eq:EEst:id:out}), we note that $\pa_t$ is killing, which gives $\pi_{\mu\nu}^{\pa_t}=0$. (\ref{eq:EEst:id:out}) can be derived by applying in the region $\mathcal{D}_{u_1}^{-u_2}$  the energy identity \eqref{energyeq} with  $X=\pa_t$, $Y=0$, $\chi=0$.

\begin{remark}
In application to the proof of our main theorem   we will  take  either  $\G=0$ or $\phi=0$  which allows us  to control energy fluxes for  Maxwell field and the scalar field separately.
\end{remark}
\subsection{The decay of energy flux}
In this subsection, we derive the decay of  energy flux and the weighted energy flux  in the exterior region. To obtain bounds for the energy flux, in this paper, we will constantly employ the following Gronwall's type inequality.

\begin{lem}\label{1.6.4.18}
Let $f_0$ be a nonnegative non-increasing function,  $a, b >1$ and $R\ge 1$. Let $f$ and $g$ be two nonnegative continuous functions such that for any $\f12 R\le \tau\le \iota,$ there holds
\begin{equation}\label{1.6.1.18}
f(\tau,\iota)+g(\tau,\iota)\les f_0(\tau)+\int_\tau^\iota f(\tau', \iota) {\tau'}^{-a}d\tau'+\int_\tau^{\iota} {\iota'}^{-b} g(\tau,\iota') d\iota'.
\end{equation}
If $g(\tau, \iota)$ is non-increasing with respect to $\tau$ or $f(\tau, \iota)$ is non-decreasing with respect to $\iota$, then
\begin{equation}\label{1.6.3.18}
f(\tau, \iota)+g(\tau, \iota)\les f_0(\tau),  \qquad \f12 R \le \tau\le \iota.
\end{equation}
\end{lem}

\begin{proof}
We only consider the case that $g(\tau, \iota)$ is non-increasing with respect to $\tau$, since the other case is similar.
For any fixed $\iota_0\ge \tau_0\ge \frac{1}{2} R$, from (\ref{1.6.1.18}) and the property of $f_0$ and $g$
we can derive for $ \tau_0\le \tau\le \iota\le \iota_0$ that
\begin{align*}
f(\tau, \iota)+g(\tau, \iota)&\les f_0(\tau_0)+\int_{\tau_0}^{\iota_0} g(\tau_0, \iota') {\iota'}^{-b} d\iota'
+ \int_\tau^\iota (f+g)(\tau', \iota) {\tau'}^{-a} d\tau',
\end{align*}
where we also used $g$ is nonnegative.
Applying the standard  Gronwall inequality to  $f(\cdot, \iota)+g(\cdot, \iota)$,  we can obtain
\begin{align*}
f(\tau, \iota)+g(\tau, \iota)&\les  f_0(\tau_0)+\int_{\tau_0}^{\iota_0} g(\tau_0, \iota') {\iota'}^{-b} d\iota' \quad
\mbox{ for }  \tau_0 \le \tau\le \iota\le \iota_0,
\end{align*}
which in particular implies that
\begin{equation}\label{1.26.1.18}
f(\tau_0, \iota_0)+g(\tau_0, \iota_0)\les  f_0(\tau_0)+\int_{\tau_0}^{\iota_0} g(\tau_0, \iota') {\iota'}^{-b} d\iota', \quad
\frac{1}{2}R \le \tau_0 \le \iota_0.
\end{equation}
Applying the standard  Gronwall inequality once again to the function $f(\tau_0, \cdot)+g(\tau_0, \cdot)$, we can derive
$f(\tau_0, \iota_0)+g(\tau_0, \iota_0)\les f_0(\tau_0)$ for $\frac{1}{2} R\le \tau_0 \le \iota_0$.
Since $\tau_0$ and $\iota_0$ are arbitrary, we obtain the desired estimate (\ref{1.6.3.18}).
\end{proof}
%For the application, we will let $\tau=-u_1$ and $\iota=-u_2$.

Recall that the chargeless part of the Maxwell field
\[
\tilde{F}=F-q_0 r^{-2}\chi_{\{t+R\leq r\}}dt\wedge dr.
\]
%Note that by the definition of the charge, $q_0\les\E_{0,\ga_0}$. Since the data can be sufficiently small, $q_0^2\le 1$.

For solution $(F, \phi)$ of the massive Maxwell Klein-Gordon equations, it is easy to check that in the exterior region we have
\begin{equation}\label{1.6.5.18}
\Box_A\phi-\phi=0,\quad \pa^\nu \tilde{F}_{\mu\nu}=J[\phi]_\mu.
\end{equation}
\begin{prop}
\label{prop:Energyflux:decay:0}
For the solution $(F, \phi)$ of the massive MKG equations, we have the following estimates on energy flux
\begin{equation}
\label{eq:EEst:decay:0}
E[\phi, \tilde{F}](\H_{u_1}^{-u_2})+E[\phi, \tilde{F}](\Hb_{-u_2}^{u_1})\les (u_1)_+^{-\ga_0}\mathcal{E}_{0, \ga_0}
\end{equation}
as well as the estimates for the weighted energy flux
\begin{equation}
\label{eq:EEst:decay:0:p1}
\begin{split}
\int_{\H_{u_1}^{-u_2}}r(|D_L\phi|^2+|\tilde\a|^2) +\int_{\Hb_{-u_2}^{u_1} }r(|\D\phi|^2+\tilde{\rho}^2+\tilde\si^2+|\phi|^2) \les (u_1)_+^{1-\ga_0}\mathcal{E}_{0, \ga_0}
\end{split}
\end{equation}
for all $u_2\leq u_1\leq -\frac{R}{2}$.
\end{prop}
\begin{proof}
By definition of $\tilde{F}$, we can obtain that
\begin{align*}
F_{0\mu} J[\phi]^\mu- \tilde{F}_{0\ga}J^\ga+\Re(h\cdot \overline{D_0\phi})=(F_{0\mu}-\tilde{F}_{0\mu})J[\phi]^\mu=q_0 r^{-2}\Im(\phi\cdot \overline{D_{\pa_ r} \phi}).
\end{align*}
Apply the energy estimate \eqref{eq:EEst:id:out} to $(\tilde{F}, \phi)$. By using Cauchy-schwarz inequality,  we can derive that
\begin{align*}
&E[\phi, \tilde{F}](\H_{u_1}^{-u_2})+E[\phi, \tilde{F}](\Hb_{-u_2}^{u_1})\\
&\les E[\phi, \tilde{F}](\Si_0^{u_1, u_2})+\iint_{\mathcal{D}_{u_1}^{-u_2}} r^{-2}|D\phi||\phi|\\
&\les (u_1)_+^{-\ga_0}\mathcal{E}_{0, \ga_0}+\iint_{\mathcal{D}_{u_1}^{-u_2}} r^{-2}(|D_L\phi|^2+|\D\phi|^2+|D_{\Lb}\phi|^2+|\phi|^2)\\
&\les (u_1)_+^{-\ga_0}\mathcal{E}_{0, \ga_0}+\int_{u_1}^{u_2} u_+^{-2}E[\phi, \tilde{F}](\H_u^{-u_2})du+\int_{-u_1}^{-u_2}v_+^{-2}E[\phi, \tilde{F}](\Hb_{v}^{u_1})dv,
\end{align*}
where we  used that in the exterior region $\{t+R\leq r\}$, we always have $v\leq r\leq 2v$. Applying Lemma \ref{1.6.4.18} with $$\tau=-u_1, \iota=-u_2,\, E[\phi, \tilde F](\H_u^{-u_2})=f(-u, -u_2), \, E[\phi, \tilde F](\Hb_{v}^{u_1})=g(-u_1, v), $$ we  derive the energy flux decay estimate \eqref{eq:EEst:decay:0}.

 To see (\ref{eq:EEst:decay:0:p1}), we apply the $r$-weighted energy identity \eqref{pWescaout} to $(\tilde{F}, \phi)$ with $p=1$. First we note that
 \begin{align*}
F_{X\mu} J[\phi]^\mu- \tilde{F}_{X\ga}J^\ga+\Re(h\cdot \overline{D_X\phi+\chi \phi})=r(F_{L\Lb}-\tilde{F}_{L\Lb})J^{\Lb}[\phi]=-\f12 q_0 r^{-1}J_L[\phi]
\end{align*}
as $h$ vanishes in this case and other components of $F-\tilde{F}$ vanish as well.
This implies
\begin{align*}
  &\f12\iint_{\mathcal{D}_{u_1}^{- u_2}}(r^{-2}|D_L(r\phi)|^2+|\tilde\a|^2+|\D\phi|^2+\tilde{\rho}^2+\tilde\si^2)\\
&+\int_{\H_{u_1}^{-u_2}}r(r^{-2}|D_L(r\phi)|^2+|\tilde\a|^2)+\int_{\Hb_{-u_2}^{u_1} }r(|\D\phi|^2+\tilde{\rho}^2+\tilde\si^2+|\phi|^2)\\
\les & (u_1)_+^{1-\ga_0}\mathcal{E}_{0, \ga_0}+\iint_{\mathcal{D}_{u_1}^{-u_2}}(|\phi|^2+ r^{-1}|D_L \phi||\phi|)\\
\les & (u_1)_+^{1-\ga_0}\mathcal{E}_{0, \ga_0}+\int_{u_1}^{u_2}E[\phi, \tilde{F}](\H_u^{-u_2})du\les (u_1)_+^{1-\ga_0}\mathcal{E}_{0, \ga_0},
\end{align*}
where we employed the first  estimate in (\ref{eq:EEst:decay:0}) to derive the last inequality.
 Thus the proof of  the weighted energy flux decay estimate \eqref{eq:EEst:decay:0:p1} is completed.
\end{proof}

\subsection{Commutators}\label{1.9.1.18}
In Proposition \ref{prop:Energyflux:decay:0}, we established estimates for the energy flux and weighted energy flux of the  lowest order. For higher order estimates we  make use of  the commuting vector field approach.
The Killing  vector fields that will be used as commutators are the generators of the Poincar\'e group
\[
\Gamma=\{\pa, \Om_{\mu\nu}=x^\mu\pa_\nu-x^\nu\pa_\mu\},
\]
where $x^0=-t$.   It is convenient to define   the following signature   of     $\pa$ and $\Om$,
       capturing their  different  weights  in $x$:
    \[
\zeta(\pa_\mu)=-1,\quad \zeta(\Om_{\mu\nu})=0.
\]
For $Z^k=\Pi_{i=1}^{k}Z_i$ with $Z_i\in \Gamma$, we have $\zeta(Z^k)=\sum\limits_{i=1}^k \zeta(Z_i)$. In particular, $\zeta(Z^0)=0$.
For convenience  we denote by $\Ga^k=\{Z^k=Z_1 \cdots Z_k,\, Z_i\in \Ga\}$ and $\Ga^0$ is $\{Z^0\}$.

To derive pointwise  decay  estimate  we need energy decay estimates for  higher order   derivatives of  solutions with respect to  any vector-field $Z\in \Ga$.
For the scalar field $\phi$, it is natural to take the
covariant derivative $D_Z\phi= Z^\mu D_\mu \phi= Z^\mu( \pr_\mu \phi+\sqrt{-1} A_\mu \phi) $, associated to the connection field $A$,   while  for  the Maxwell field $F$ and $1$-form $J$   we take  the Lie derivative
\begin{align*}
 (\mathcal{L}_Z F)_{\mu\nu}&=Z(F_{\mu\nu})-F(\mathcal{L}_Z \pa_{\mu}, \pa_\nu)-F(\pa_\mu, \mathcal{L}_{Z}\pa_\nu),\\
(\mathcal{L}_Z J)_\mu&= Z(J_\mu)-J(\mathcal{L}_Z \pa_\mu).
\end{align*}
  We  record  the following useful  commutator identities:
\begin{lem}
 \label{lem:commutator}
For any killing vector field $Z$, we have
\begin{align*}
[\Box_A-1, D_Z]\phi&=2i Z^\nu F_{\mu \nu}D^\mu \phi+i \pa^\mu(Z^\nu F_{\mu\nu})\phi,\\
\pa^\mu(\mathcal{L}_Z G)_{\mu\nu}&=(\mathcal{L}_Z \delta G)_\nu
\end{align*}
for any complex scalar field $\phi$ and any closed 2-form $G$. Here $ \delta G$ is the 1-form defined by $(\delta G)_\nu=\pa^\mu G_{\mu\nu}$.
\end{lem}
 \begin{proof}
 The proof is standard,  see  for example \cite{yangMKG}.
 \end{proof}
 According to the above lemma,  the Maxwell field commutes    with the Lie derivatives  along  Killing fields.  To treat the commutator  error terms   generated by  $\Box_A-1$,  it is helpful  to define a trilinear form for any two form $\G$, any complex scalar field $f$ and  killing vector field $Z\in \Ga$,
\begin{equation}\label{12.9.1.17}
Q(\G, f, Z)=2i Z^\nu \G_{\mu \nu}D^\mu f+i \pa^\mu(Z^\nu \G_{\mu\nu})f.
\end{equation}
 We have the following double commutator identity (see \cite[Lemma 9]{yangMKG}).
\begin{lem}
\label{lem:Est4commu:id:2}
For all $X, \, Y\in \Ga$, we have
\begin{equation}
\label{eq:Est4commu:id:2}
\begin{split}
[D_Y, [\Box_A-1, D_{X}]]\phi= Q(\cL_Y F, \phi, X)+Q(F, \phi, [Y, X])-2 F_{Y\mu}F^{\mu}_{\;X}\phi.
\end{split}
\end{equation}
%Here we note that $\mathcal{L}_Z F=\mathcal{L}_Z d A=d\mathcal{L}_Z A$.
\end{lem}
This lemma gives the structure for the error terms of second order derivatives for the scalar fields, which will be used together with the following formula.
\begin{lem}\label{cor1}
Given two vector fields  $X, Y\in \Ga$, we have
\begin{align*}
&(\Box_A-1)D_X D_Y \phi=Q(F, D_Y\phi, X)+Q(F, D_X\phi, Y)+[D_X,[\Box_A-1, D_Y]]\phi+D_X D_Y((\Box_A-1)\phi).
\end{align*}
\end{lem}
\begin{proof}
In fact we can compute that
\begin{align*}
&(\Box_A-1)D_X D_Y \phi\\
&=[\Box_A-1, D_X]D_Y \phi+D_X((\Box_A-1)D_Y\phi)\\
&=Q(F, D_Y\phi, X)+D_X([\Box_A-1,D_Y]\phi+D_Y((\Box_A-1)\phi))\\
&=Q(F, D_Y\phi, X)+[D_X,[\Box_A-1,D_Y]]\phi+[\Box_A-1,D_Y] D_X\phi+D_X D_Y((\Box_A-1)\phi)\\
&=Q(F, D_Y\phi, X)+Q(F, D_X\phi, Y)+[D_X,[\Box_A-1, D_Y]]\phi+D_X D_Y((\Box_A-1)\phi).
\end{align*}
\end{proof}
Therefore, in view of (\ref{eq:Est4commu:id:2}), if $\phi$ solves the equation $\Box_A\phi=\phi$, the righthand side of the above identities are linear combination of the trilinear forms of $Q$ except the cubic term $F_{Y\mu}F^{\mu}_{\;X}\phi$.

The next lemma shows that $Q(\G, \phi, Z)$ indeed verifies  a  null type  structure, that is the ``bad" component $\ab[\tilde F]$ does not interact with $D_{\Lb}\phi$.
\begin{lem}
\label{lem:Est4commu:1}
For any $2$-form $\G=(\a, \rho, \si, \underline{\a})$    and a complex scalar field $\phi$, in the exterior region $\{t+R\leq r\}$ we have
\begin{equation}
\label{eq:Est4commu:1}
\begin{split}
|Q(\G, \phi, Z)|&\les r^{\zeta(Z)+1}(|\a||D\phi|+|\ud\G||D_L\phi|+|\si||\D\phi|)+u_+^{\zeta(Z)+1}(|\rho||D_{\Lb}\phi|+|\ab||\D\phi|)\\
&+\left(u_+^{\zeta(Z)+1}|J_{\Lb}|+r^{\zeta(Z)+1}(|J_L| +|\J|)+r^{\zeta(Z)}|\G|\right)|\phi|
\end{split}
\end{equation}
for all $Z\in \Gamma$, where  $\ud\G$ denotes all the components of $\G$ except $\a[\G]$.   Here   $(J_L, J_{\Lb}, \J)$, $\J=(J_{e_1}, J_{e_2})$     denote    the  null  components    of  the 1-form $J=-\delta \G$.
\end{lem}
\begin{proof}
Recall the definition of $Q$ in (\ref{12.9.1.17}), we have in view of (\ref{EQMKG}) that
\begin{equation*}
Q(\G, \phi, Z)=2i Z^\nu \G_{\mu\nu}D^{\mu}\phi+i (\pa^\mu Z^\nu \G_{\mu\nu}-J_Z)\phi.
\end{equation*}
When $Z=\pa_\mu$, we can estimate that
\begin{align*}
|Q(\G, \phi, \pa_\mu)|\les (|\a|+|\ab|+|\si|)|\D\phi|+(|\rho|+|\ab|)|D_L\phi|+|D_{\Lb}\phi|(|\rho|+|\a|)+|J||\phi|.
\end{align*}
As in this case $\zeta(\pa_\mu)=-1$,  \eqref{eq:Est4commu:1} holds  indeed for $Z=\pa_\mu$.

\medskip

When $Z=\Om_{ij}$, we can easily  show that
\begin{align*}
|Q(\G, \phi, \Om_{ij})|\les r(|\a||D_{\Lb}\phi|+|\ab||D_L\phi|+|\si||\D\phi|+|\J||\phi|)+|\G||\phi|.%This could be improved if the charge is a problem |F| could be replaced by |\a|+|\ab|+|\si|
\end{align*}
For $Z=\Om_{0i}$, we  can write that
\[
\Om_{0i}=t\pa_i+x_i\pa_t=\om_i(t\pa_r+r\pa_t)+t(\pa_i-\om_i\pa_r)=\om_i(vL-u\Lb)+t(\pa_i-\om_i\pa_r).
\]
Therefore we can
 bound that
\begin{align*}
|Q(\G, \phi, \Om_{0j})|\les & r(|\a||D\phi|+(|\rho|+|\ab|)|D_L\phi|+|\si||\D\phi|)+u_+(|\rho||D_{\Lb}\phi|+|\ab||\D\phi|)\\&+(r|\J|
+r|J_L|+u_+|J_{\Lb}|+|\G|)|\phi|.
\end{align*}
Since $\zeta(\Om)=0$, we conclude that the estimate \eqref{eq:Est4commu:1} hold for all $Z\in\Gamma$.
\end{proof}
Our next lemma shows that the quadratic  part of the cubic terms $F_{Y\mu}F^{\mu}_{\;X}\phi$ also posses a  necessary null structure. The following result will be used in Section \ref{Sec1.15}.
\begin{lem}
\label{lem:Est4commu:2}
For all $X, \, Y\in \Ga$ and $r\geq R$, we have the following estimate
\begin{equation}
\label{eq:Est4commu:2}
\begin{split}
|F_{Y\mu}F^{\mu}_{\;X}|\les  &u_+^{\zeta(XY)+2}|\ab|^2+u_+^{\zeta(XY)} r^{2}(|\si|^2+|\a|^2+|\rho|^2+|\a||\ab|).
\end{split}
\end{equation}
%Here we note that $\mathcal{L}_Z F=\mathcal{L}_Z d A=d\mathcal{L}_Z A$.
\end{lem}
\begin{proof}
If both $X, \, Y\in\{\pa_\mu\}$, we can simply bound that
\[
|F^{\mu}_{\ X}F_{Y\mu}|\les |F|^2.
\]
For this case $\zeta(XY)=-2$, in particular estimate \eqref{eq:Est4commu:2} holds.

If both $X, \, Y \in \{\Om_{\mu\nu}\}$, we can write the Lorentz boost as in the previous lemma
\[
\Om_{0j}=\om_j(vL-u\Lb)+t(\pa_j-\om_j\pa_r).
\]
Then we can bound that
\begin{align*}
|F^{\mu}_{\ X}F_{Y\mu}|&\les r(|\rho|+|\ab|)(u_+|\rho|+r|\a|)+(u_+|\ab|+r|\a|+r|\si|)^2\\
&\les u_+^2|\ab|^2+r^2(|\a|^2+|\si|^2+|\rho|^2+|\a||\ab|). %This can also be improved if neccessary.
\end{align*}
Since $\zeta(XY)=0$ in this case, we conclude that estimate \eqref{eq:Est4commu:2} is proved.

If, without loss of generality, $X=\pa_\mu$, $Y=\Om_{\nu\ga}$, we then can show that
\begin{align*}
|F^{\mu}_{\ X}F_{Y\mu}|&\les r|\a||F|+r|\rho|(|\rho|+|\ab|)+r|\si|(|\ab|+|\si|)+u_+|\ab|(|\a|+|\ab|+|\si|).
\end{align*}
Considering that $\zeta(XY)=-1$ for this situation, by using Cauchy-Schwarz's inequality, we have shown that estimate \eqref{eq:Est4commu:2} holds for all $X$, $Y\in \Gamma$.
\end{proof}

\section{Proof of main theorem}
In this section, we prove Theorem \ref{m.thm} by a bootstrap argument.

\subsection{Bootstrap argument}
For some small positive constant $\dn$, verifying  $1>\dn\ge \E_{2,\ga_0}$, to be determined later,
we make a set of  bootstrap assumptions on the the solution $(F, \phi)$ to the massive MKG equations in the exterior region.  The bootstrap assumptions mainly consist of the higher order energy flux decay as well as the $r$-weighted energy flux decay of the solution. For the scalar field, the highest order $r$-weighted energy estimates can at most have weights $r$ while for the chargeless part of the Maxwell field, the weights are chosen to be $r^{\ga_0}$ for some $\ga_0>1$.

 Let $v_*>\frac{R}{2}$ be a fixed constant.  We suppose  the following estimates hold
\begin{align}
\label{eq:BT:energy}
&E[D_Z^k\phi, \cL_Z^k\tilde{F}](\H_{u_1}^{-u_2})+E[D_Z^k\phi, \cL_Z^k\tilde{F}](\Hb_{-u_2}^{u_1})\leq 2 (u_1)_+^{-\ga_0+2\zeta(Z^k)}\dn,\\
\label{eq:BT:PWE:scal}
&\int_{\H_{u_1}^{-u_2}}r|D_L D_Z^k\phi|^2 +\int_{\Hb_{-u_2}^{u_1} }r(|\D D_Z^k\phi|^2+|D_Z^k\phi|^2)  \leq 2(u_1)_+^{1-\ga_0+2\zeta(Z^k)}\dn,\\
\label{eq:BT:PWE:Max}
&\int_{\H_{u_1}^{-u_2}}r^{\ga_0}|\a[\cL_Z^k \tilde{F}]|^2 +\iint_{\mathcal{D}_{u_1}^{-u_2}}r^{\ga_0-1}|(\a, \rho,\sigma)[\cL_Z^k\tilde{F}]|^2\\
\notag
&\qquad \qquad+\int_{\Hb_{-u_2}^{u_1} }r^{\ga_0}(|\rho[\cL_Z^k\tilde{F}]|^2+|\si[\cL_Z^k\tilde{F}]|^2) \leq 2(u_1)_+^{2\zeta(Z^k)}\dn
\end{align}
for all $-v_*\le u_2<u_1\leq -\frac{R}{2}$, $Z^k\in \Ga^k$, $k\leq 2$.

Here we remark that $k=0$, due to $\zeta(Z^0)=0$, (\ref{eq:BT:energy}) and (\ref{eq:BT:PWE:scal}) can be improved in view of Proposition \ref{prop:Energyflux:decay:0}. The improvement of all other estimates relies on the bootstrap assumptions. At the end of this section, we will be able to show the inequalities (\ref{eq:BT:energy})-(\ref{eq:BT:PWE:Max}) hold with $2\dn$ replaced by $C(\E_{2,\ga_0}+\dn^2)$. We then will choose $\dn$ to achieve the improvements of the bootstrap assumptions. We note that the improved estimates are independent of the choice of $v_*$, which allows us to achieve Theorem \ref{m.thm} (2) for any $v_*> R/2$ by the principle of  continuation.

The strategy of the proof is that with these assumptions, we first prove the pointwise decay estimates for the scalar field $\phi$ as well as the chargeless part of the Maxwell field $\tilde{F}$. We then obtain energy estimates and the $r$-weighted energy estimates, which improves the above bootstrap assumptions due to the smallness of the initial data. Combined with the standard local existence result, we conclude that the solution exists globally in time in the exterior region. As a consequence, we also obtain the decay estimates for the solutions.

We will rely on Lemma \ref{lem:globalSobolev} and a similar Sobolev inequality to derive the pointwise decay. For this purpose, we need to bound the initial value of the solution $(F, \phi)$ and its derivatives in terms of the given data, which is treated below.
\begin{lem}
Let $2\le p\le 4$ and $u\le -\f12R$. For the admissible data  $(F,\phi)$ of (\ref{EQMKG}) we have the following.
\begin{itemize}
\item[(1)] For all $Z^l\in \Ga^l$ and $l\le 2$,
\begin{equation}\label{12.22.1.17}
\int_{{\mathbb S}^2}|D_Z^l \phi|^p(u, -u, \omega) d\omega\les u_+^{-2+\frac{p}{2}(-\ga_0+2\zeta(Z^l))}\E_{l,\ga_0}^\frac{p}{2}.
\end{equation}
In particular, for all $Z^l\in \Ga^l $ and $l\le 1$,  
\begin{equation}\label{1.31.2.18}
\int_{{\mathbb S}^2}|D_Z^l \phi|^p(u, -u, \omega) d\omega\les u_+^{{2-2p}+\frac{p}{2}(-\ga_0+2\zeta(Z^l))}\E_{l+1,\ga_0}^\frac{p}{2}.
\end{equation}
\item[(2)]  For all  $Z=\Omega_{ij}$ and $l\le 1$,
\begin{equation}
\int_{{\mathbb S}^2}|\Lie_Z^l \tilde F|^p(u,-u, \omega)d\omega\les u_+^{2-2p-\frac{p}{2}\ga_0} \E_{l+1,\ga_0}^\frac{p}{2}.\label{12.26.2.17}
\end{equation}

\end{itemize}
\end{lem}
\begin{proof}
Let us consider (\ref{12.22.1.17}) first.
  Note that we have the embedding
\begin{align}\label{12.21.3.17}
\left(\int_{\Sigma_0^u}|D_Z^l \phi|^6 dx \right)^{\frac{1}{3}}\les \int_{\Sigma_0^u} (|D_{\pa r}D_Z^l\phi|^2+|\D D_Z^l\phi|^2+|D_Z^l\phi|^2) dx \les \E_{l,\ga_0}.
\end{align}
 Since $\E_{l,\ga_0}$  is finite,  there holds
\begin{align*}
\liminf\limits_{u_+\rightarrow \infty}\int_{{\mathbb S}^2}|D_Z^l\phi|^6(u, -u, \om)d\om =0, \forall\, l\le 2.
\end{align*}
By interpolation, for all $l\leq 2$ and $2\leq p\leq 6$, we can obtain
\begin{align*}
\liminf\limits_{u_+\rightarrow \infty}\int_{{\mathbb S}^2}|D_Z^l \phi|^p(u, -u, \om) d\om =0.
\end{align*}
Note that for $l\le 2$, by integrating from the spatial infinity,  we have
\begin{equation}\label{1.31.1.18}
\int_{{\mathbb S}^2}|D_Z^l \phi|^2(u, -u, \omega) d\omega \les\int_{-2u}^\infty\int_{{\mathbb S}^2} |D_{\pa r} D_Z^l \phi| |D_Z^l \phi| dr d\omega\les u_+^{-\ga_0-2+2\zeta(Z^l)} \E_{l,\ga_0}
\end{equation}
and if $l \le 1$,  we can derive
\begin{equation}\label{1.31.3.18}
\int_{{\mathbb S}^2}|D_Z^l \phi|^2(u, -u, \omega) d\omega \les u_+^{-3}\|r D_{\pa r} D_Z^l \phi\|_{L^2(\Sigma_0^u)}\|D_Z^l \phi\|_{L^2(\Sigma_0^u)}\les u_+^{-\ga_0-3+2\zeta(Z^l)}\E_{l+1, \ga_0}.
\end{equation}
Next, we consider the $L^4$ estimates. Similar to the proof of (\ref{12.25.1.17}) in  Lemma \ref{lem:globalSobolev},  we can obtain for a complex scalar field $f$ that
\begin{equation*}
\int_{S(u, -u)}| f|^4 r^{-2}
\les \lim_{u_+\rightarrow\infty} \int_{S(u, -u)}| f|^4 r^{-2} + \int_{\Sigma_u^0}|D_{\pa r} f|^2 r^{-2} dx\cdot \int_{\Sigma_0^u}(| f|^2+r^2|\D  f|^2)  r^{-2} dx.
\end{equation*}
Applying the above inequality to $f=D_Z^l \phi, $ with $l\le 2$,  gives 
\begin{equation*}
\int_{{\mathbb S}^2}| D_Z^l \phi|^4(u, -u, \omega)\les  u_+^{-2}\sum_{k\le 1}\|\bar D^k D_Z^l \phi\|_{L^2(\Sigma_u^0)}^4\les u_+^{-2-2\ga_0+4\zeta(Z^l)} \E^2_{l,\ga_0}
\end{equation*}
where  $\bar D$ is the projection of $D$ to ${\mathbb R}^3\times \{t=0\}$.

When $l\le 1$,  we have the improved estimate
\begin{align*}
\int_{{\mathbb S}^2}|D_Z^l \phi|^4(u, -u, \omega)\les u_+^{-6}\|r D_{\pa r} D_Z^l f\|^2_{L^2(\Sigma_0^u)}&\sum_{k\le 1, 1\le i<j\le 3}\|D^k_{\Omega_{ij}}D_Z^l \phi\|^2_{L^2(\Sigma_0^u)}\\
&\les u_+^{-6-2\ga_0+4\zeta(Z^l)} \E^2_{l+1, \ga_0}.
\end{align*}
where we used the fact that $r|\D f|\les \sum_{1\le i<j\le 3}|D_{\Omega_{ij}} f|$ for any complex scalar field.  
By interpolating  the above two estimates with (\ref{1.31.1.18}) and (\ref{1.31.3.18}) respectively, 
the estimates in (1) of this lemma can be proved.

To see (\ref{12.26.2.17}), we first note that it is direct to check $|\Lie_Z^l \tilde F|\approx |\Lie_Z^l {\tilde E}_i|+|\Lie_Z^l {\tilde H}_i|$. Then it suffices to show (\ref{12.26.2.17}) holds for $\Lie_Z^l \tilde E_i$ and $\Lie_Z^l \tilde H_i$ with $Z=\Omega_{ij}$ and $l\le 1$. We then can repeat the proof for (\ref{1.31.2.18}) to obtain the estimate in (\ref{12.26.2.17}) to $\Lie_Z^l \tilde E$ and $\Lie_Z^l \tilde H$ separately. Thus we can complete the proof of (\ref{12.26.2.17}).

\end{proof}
\subsection{Decay estimates for the scalar field}
 We first derive the decay estimates for the scalar field with the help of the bootstrap assumptions. 
\begin{prop}
\label{prop:decay:scal}
Under the  bootstrap assumptions \eqref{eq:BT:energy} and \eqref{eq:BT:PWE:scal}, in $\{r\ge t+R, v\le v_*\}$, we have the decay estimates for the scalar fields
\begin{align}
\label{eq:decay:scal:pt}
r^2|\D\phi|^2+u_+^2|D_{\Lb}\phi|^2+r^2|D_L \phi|^2 +u_+^{-2\zeta(Z)}|D_Z\phi|^2&\les \dn r^{-\frac{5}{2}+\ep}u_+^{\f12-\ga_0},\\
\label{eq:decay:scal:pt:phi}
|\phi|^2 &\les \dn r^{-3}u_+^{-\ga_0},
\end{align}
where $Z\in \Ga$.
\end{prop}
\begin{proof}
Let's first consider the decay estimates for $D_Z\phi$ for all $Z\in \Ga$. We will make a slight modification on Lemma \ref{lem:globalSobolev}.
We can show that, by integrating from the initial slice,
\begin{align*}
|\int_{{\mathbb S}^2} &r^{\frac{5}{2}}|D_Z^l\phi|^p(u, v, \om) d\om -\int_{{\mathbb S}^2} r^{\frac{5}{2}}|D_Z^l \phi|^p(u, -u, \om) d\om|\\
&\les \int_{-u}^{v}\int_{{{\mathbb S}^2}}r^\frac{5}{2}|D_L D_Z^l\phi| |D_Z^l\phi|^{p-1}   dv' d\om+\int_{-u}^v\int_{{\mathbb S}^2} r^\frac{3}{2} |D_Z^l \phi|^p dv' d\om\\
&\les (\|r^\frac{1}{2} D_L D_Z^l \phi\|_{L^2(\H_u^{-u_2})}+\|r^{-\f12} D_Z^l \phi\|_{L^2(\H_u^{-u_2})})\||D_Z^l \phi|^{p-1}\|_{L^2(\H_u^{-u_2})}.
%&\les  \dn^\f12 r^{-\frac{5}{2}}u_+^{\f12(1-\ga_0+2\zeta(Z^l))}\left(\int_{-u}^{v}\int_{\om}|D_Z^l\phi|^{2p-2} r^{2}dvd\om\right)^{\f12}.
\end{align*}
Here $Z^l\in \Ga^l$ with $l\le 2$.

When $p=2$, we can bound the term of $\||D_Z^l \phi|^{p-1}\|_{L^2(\H_u^{-u_2})}$ by the energy flux through $\H_u$.  While for $p=4$, similar to  (\ref{12.21.3.17}) we can bound the integral by using Sobolev embedding on the outgoing null hypersurface $\H_u$ as follows,
\begin{align*}
(\int_{-u}^{-u_2}\int_{\mathbb{S}^2} r^2 |D_Z^l \phi|^6 dv d\omega )^\frac{1}{3}&\les \int_{\H_u^{-u_2}}\{|D_L D_Z^l \phi|^2+|{\sl D} D_Z^l \phi|^2+|D_Z \phi|^2 \}.
\end{align*}
Also by using (\ref{12.22.1.17}), we can obtain for $l\le 2$,
\begin{align*}
&\int_{{\mathbb S}^2} r^{\frac{5}{2}} |D_Z^l \phi|^2(u,v, \omega)d\omega\les u_+^\f12 u_+^{-\ga_0+2\zeta(Z^l)} \dn, \\
&\int_{{\mathbb S}^2} r^\frac{5}{2} |D_Z^l \phi|^4(u,v, \omega) d\omega \les u_+^\f12(u_+^{-\ga_0+2\zeta(Z^l)} \dn)^2.
\end{align*}
 Then interpolation implies that for $2\le p\le 4$, $l\le 2$,
\begin{align*}
\int_{{\mathbb S}^2}|D_Z^l\phi|^p(u, v, \om) d\om &\les  \|D_Z \phi(u,v, \cdot)\|_{L_\omega^2}^{4-p}\|D_Z\phi(u,v, \cdot)\|_{L_\omega^4}^{2p-4}\\
&\les \dn^{\frac{p}{2}}r^{-\frac{5}{2}}u_+^{\f12+\frac{p}{2}(-\ga_0+2\zeta(Z^l))}.
\end{align*}
Take $Z^l=\Om_{ij} Z$ and $p=2+\ep$. By using Sobolev embedding on the unit sphere we obtain the decay estimates for
$D_Z\phi$,
\begin{equation*}
|D_Z\phi|^2\les \dn r^{-\frac{5}{2}+\ep}u_+^{\f12 +2\zeta(Z)-\ga_0}.
\end{equation*}
Now by taking $Z=\Om_{ij}$, we  can derive the decay estimate for $|\D\phi|$ which could be bounded above by the sum of $r^{-1}|D_{\Om_{ij}}\phi|$. For the other components such as $D_{L}\phi$ and $D_{\Lb}\phi$, first we can take $Z=\pa_\mu$ to conclude that
\[
|D\phi|^2\les \dn r^{-\frac{5}{2}+\ep}u_+^{\f12-\ga_0-2}.
\]
In particular, we have
\[
u_+^2|D_{\Lb}\phi|^2\les \dn r^{-\frac{5}{2}+\ep}u_+^{\f12-\ga_0}.
\]
For $D_L\phi$, we use the Lorentz boost $Z=\Om_{0j}$. Notice that
\[
\Om_{0j}=\om_j (vL-u\Lb)+t(\pa_j-\om_j\pa_r).
\]
In the exterior region $\{t+R\leq r\}$, we have
\[
|D_{t(\pa_j-\om_j\pa_r)}\phi|^2\les \sum\limits_{l, k}|D_{\Om_{lk}}\phi|^2.
\]
Therefore we can show that
\begin{align}
r^2|D_L\phi|^2\les |D_{vL}\phi|^2\les \sum\limits_{j}|D_{\Om_{0j}}\phi|^2+|D_{u\Lb}\phi|^2+\sum\limits_{l, k}|D_{\Om_{lk}}\phi|^2\les \dn r^{-\frac{5}{2}+\ep}u_+^{\f12-\ga_0}.\label{1.1.1.18}
\end{align}
Therefore the decay estimate \eqref{eq:decay:scal:pt} holds.

To show the improved decay estimate \eqref{eq:decay:scal:pt:phi} for $\phi$, we need the associated improved energy flux decay. Recall the bootstrap assumption (\ref{eq:BT:energy}) implies that
\begin{align*}
\int_{\H_u^{-u_2}}|D_Z^l\phi|^2\les E[D_Z^l \phi](\H_u^{-u_2}) \les\dn u_+^{-\ga_0 +2\zeta(Z^l)},\quad \forall\, l\leq 2.
\end{align*}
Therefore similar to (\ref{1.1.1.18}), we can show that
\begin{align*}
\int_{\H_u^{-u_2}}r^2|D_L D_Z^l\phi|^2  &\les \int_{\H_u^{-u_2}}(|D_{\Om_{0j}} D_Z^l\phi|^2+|D_{\Om_{ij}}D_Z^l\phi|^2+u_+^2 |DD_Z^l\phi|^2)\\
&\les \dn u_+^{-\ga_0 +2\zeta(Z^l)},\quad \forall\, l\leq 1.
\end{align*}
Now we substitute the above two estimates  with $Z^l=\Omega_{ij}^l, l\le 2$, $\forall\, 1\le i<j\le 3$  to the following inequality, which is derived by applying Lemma \ref{lem:globalSobolev} to $f=\phi$, $\ga=\frac{3}{2}$, $\ga_0'=\ga_2=1$,
\begin{align*}
\sup\limits_{S{(u, v)}}|r^{\frac{3}{2}} \phi|^4
&\les \sum\limits_{l\leq 1, 1\le i<j\le 3} \int_{S(u, -u)}|r D_{\Om_{ij}}^l \phi|^4 +\sum\limits_{k\leq 2,1\le i<j\le 3}\int_{\H_u^{-u, v_*}} |D_{\Om_{ij}}^k f|^2\\
&\times \sum\limits_{l\leq 1, 1\le i<j\le 3 }\int_{\H_u^{-u, v_*}} r^{-1}|D_L D_ {\Om_{ij}}^l(r^{\frac{3}{2}} \phi)|^2 \les \dn^2 u_+^{-2\ga_0},
\end{align*}
%Note that
%\begin{align*}
%\int_{\H_u^{-u_2}} |D_{\Om_{ij}}^k \phi|^2\les E[D_{\Om_{ij}}^k\phi](\H_u^{-u_2})\les \dn u_+^{-\ga_0},\quad \forall \, k\leq 2.
%\end{align*}
 where the term on the sphere $S(u, -u)$ has been treated by (\ref{1.31.2.18}).  Thus we have completed the proof of \eqref{eq:decay:scal:pt:phi}.
%Indeed, we have
%\begin{align*}
%\int_{{\mathbb S}^2}|D_{\Om_{ij}}^l \phi|^4(u, -u, \om) d\om &\les u_+^{-3-\f12 \ga_0}\left(\int_{|x|\geq -2u}|DD_{\Om_{ij}}^l\phi|^2 r^{2+\ga_0} %dx\right)^\f12 \left(\int_{|x|\geq -2u}|D_{\Om_{ij}}\phi|^6 dx\right)^\f12\\
%&\les \mathcal{E}_{l+1, \ga_0}^{\f12}u_+^{-3-\f12\ga_0} \left(\int_{|x|\geq -2u}|D D_{\Om_{ij}}\phi|^2 dx\right)^{\frac{3}{2}}\\
%&\les \mathcal{E}_{l+1, \ga_0}^{2}u_+^{-6-2\ga_0}.
%\end{align*}
\end{proof}
%We remark here that we can also use the same method to derive the decay estimates for $\phi$ as $D_Z\phi$ by losing $r^{\ep}$ decay.

\subsection{Decay estimates for the Maxwell field}
In this subsection, we derive the pointwise decay estimates for the Maxwell field under the bootstrap assumptions.
\begin{prop}
\label{prop:decay:pt:Max}
In the exterior region $\{t+R\leq r, v\le v_*\}$, under the bootstrap assumptions \eqref{eq:BT:energy} and  \eqref{eq:BT:PWE:Max}, we have
\begin{align}
\label{eq:decay:Max:asirho}
|\tilde{\rho}|^2+|\a|^2+|\si|^2&\les \dn r^{-2-\ga_0}u_+^{-1},\\
\label{eq:decay:Max:ab}
|\ab|^2 &\les \dn r^{-2}u_+^{-\ga_0-1},
\end{align}
where $\tilde\rho=\rho[\tilde F]$ \begin{footnote}
{To distinguish from the component of $F$, we add the tilde to the Greek letter when it is the component of $\tilde F$.
 It is clear that   $\a=\tilde \ab$, $\tilde{\ab}=\ab$, $\tilde{\si}=\si$.}
\end{footnote}
 and clearly $|\rho|^2\les (\dn+q_0^2) r^{-2}u_+^{-2}.$
\end{prop}

Assuming Proposition \ref{prop:decay:pt:Max},  in view of $F=\tilde{F}+q_0 r^{-2}dt\wedge dr$, we can have the rough estimate
\begin{equation}\label{12.7.3.17}
|F_{X\mu}|^2 \les (\dn+q_0^2) u_+^{-2} r^{2\zeta(X)}, \,\, \forall X\in \Ga.
\end{equation}
\begin{proof}
To prove these pointwise estimates, we  apply Lemma \ref{lem:globalSobolev} with the help of the (weighted) energy flux decay through the outgoing null hypersurface $\H_u$ or the incoming null hypersurfaces $\Hb_v$.
 We first need some preliminary results. Let us show that for any  $2$-form $G$,
\begin{align}
|\Lb(\rho[G])|\les |\rho[\Lie_\pa G]|,\qquad & |\Lb (\sigma[G])|\les |\sigma[\Lie_\pa G]|,\label{12.25.3.17}\\
|{\sl\nabla}_L \a_A[G]|\les |\a[\Lie_\pa G]|,\qquad& |{\sl\nabla}_{\Lb} \ab_A[G]|\les |\ab[\Lie_\pa G]|,\label{12.25.4.17}\\
\a[\Lie_{\Omega_{ij}}^l G]= {\sl\cL}_{\Omega_{ij}}^l\a[G],\qquad&\ab[\Lie_{\Omega_{ij}}^l G]= {\sl\cL}_{\Omega_{ij}}^l\ab[G],\label{12.25.5.17} \\
\rho[\Lie_{\Omega_{ij}}^l G]= \Omega_{ij}^l\rho[G],\qquad&\sigma[\Lie_{\Omega_{ij}}^l G]= \Omega_{ij}^l\si[G].\label{12.26.1.17}
\end{align}
Indeed, we recall from \cite[Page 58-59]{LindbladMKG} that
\begin{equation}\label{1.3.1.18}
\pa_r \rho=\omega^i \rho[\Lie_{\pa_i}G], \quad \pa_r \sigma=\omega^i \sigma[\Lie_{\pa_i} G],\quad \pa_r \a_A[G]=\omega^i \a_A[\Lie_{\pa i} G].
\end{equation}
It is easy to check that
\begin{equation}\label{1.3.2.18}
\pa_t \rho[G]=\rho[\Lie_{\pa_t}G], \quad \pa_t \sigma[G]=\sigma[\Lie_{\pa_t} G], \quad \pa_t \a_A[G]=\a_A[\Lie_{\pa t} G].
\end{equation}
Hence we can obtain (\ref{12.25.3.17}) by combining the above identities.
(\ref{12.25.5.17}) and (\ref{12.26.1.17}) can be proved in view of \cite[Page 58-59 (5.20), (5.21) and (5.28)]{LindbladMKG}.

 To see (\ref{12.25.4.17}), we first note
\begin{equation*}
\sl{\nabla}_L\a_A[G]=L \a_A[G]-\a[G]_{\sl{\nabla}_L e_A}=\pa_t \a_A[G]+\pa_r\a_A[G],
\end{equation*}
where we used  the fact that $\sl{\nabla}_L e_A=0$.  Hence, in view of (\ref{1.3.1.18}) and (\ref{1.3.2.18}),  we  can derive the first inequality in (\ref{12.25.4.17}). The second inequality in (\ref{12.25.4.17}) can be proved in the same way.

  With the help of (\ref{12.25.5.17}) and  (\ref{12.26.1.17}), the estimate (\ref{12.26.2.17}) holds for the components of $ {\sl\cL}_{\Omega_{ij}}^l\a[\tilde F]$, $ {\sl\cL}_{\Omega_{ij}}^l\ab[\tilde F]$, $\Omega_{ij}^l \rho[\tilde F]$ and $\Omega_{ij}^l \si[\tilde F]$, which gives the bounds for the initial values needed in the sequel.

Let's first consider the decay estimate for $\a$. Apply Lemma \ref{lem:globalSobolev}(3)  with $H=\a$, $\ga_0'=0$ and $\ga=\ga_2=1+\f12 \ga_0$.
By using (\ref{12.25.5.17}), we can derive
\begin{align*}
\int_{\H_u^{-u_2}}|{\sl\cL}_{\Om_{ij}}^k\tilde \a|^2 r^{\ga_0} =\int_{\H_u^{-u_2}}|\a[\cL_{\Om_{ij}}^k\tilde F]|^2 r^{\ga_0}\les \dn, \quad \forall\, k\leq 2
\end{align*}
and
\begin{align*}
\int_{\H_u^{-u_2}}|{\sl{\nabla}}_L{\sl\cL}_{\Om_{ij}}^l (r^{1+\f12 \ga_0}\tilde \a)|^2 dvd\om \les\int_{\H_u^{-u_2}} (r^{\ga_0}|\a[\cL_{\pa}\cL_{\Om_{ij}}^l \tilde F]|^2+r^{\ga_0-2}|\a[\Lie_{\Om_{ij}}^l \tilde F]|^2 )\les \dn u_+^{-2}.
\end{align*}
The integral on $S(u, -u)$ is controlled in view of (\ref{12.26.2.17}). Then by using (\ref{12.25.2.17}) we obtain that
\begin{align*}
r^{2+\ga_0}|\a|^2\les \dn u_+^{-1}.
\end{align*}
For the other components, we make use of the energy flux through the incoming null hypersurface. However for $\tilde{\rho}$ and $\si$, the extra decay in $r$ relies on the $r$-weighted energy flux through the incoming null hypersurface while for $\ab$ we can only use the energy flux. Let $-\frac{R}{2}\le -u< v\le v_*.$  For $\tilde{\rho}$ and $\si$, from the bootstrap assumption \eqref{eq:BT:PWE:Max}, we first conclude that
\begin{align*}
\int_{\Hb_v^{u}}r^{\ga_0}(|\rho[\Lie_\pa\cL_{\Omega_{ij}}^k\tilde{F}]|^2+|\si[\Lie_\pa  \cL_{\Omega_{ij}}^k \tilde{F}]|^2)\les \dn u_+^{-2},\quad   \, \forall\, k\leq 1.
\end{align*}
By using (\ref{12.25.3.17}) and (\ref{12.26.1.17}), we have
\begin{equation*}
|\Lb\cL_{\Omega_{ij}}^k\rho[\tilde F]|\les |\rho[\Lie_\pa \cL_{\Omega_{ij}}^k\tilde{F}]|, \quad |\Lb\cL_{\Omega_{ij}}^k\si[\tilde{F}]|\les |\si[\Lie_\pa \cL_{\Omega_{ij}}^k\tilde{F}]|.
\end{equation*}
We therefore can derive that
\begin{align*}
\int_{\Hb_v^{u}}r^{\ga_0}\big(|\Lb(\cL_{\Omega_{ij}}^k\tilde{\rho})|^2+|\Lb (\cL_{\Omega_{ij}}^k \tilde\si)|^2\big)\les \dn u_+^{-2},\quad \forall\, k\leq 1.
\end{align*}
By using  (\ref{12.26.1.17}) and
 the bootstrap assumption \eqref{eq:BT:PWE:Max}, we have
\begin{align*}
\int_{\Hb_v^{u}}r^{\ga_0}|\Om_{ij}^k(\tilde{\rho}, \tilde\si)|^2  = \int_{\Hb_v^{u}}r^{\ga_0}|(\rho, \sigma)[\cL_{\Om_{ij}}^k\tilde F]|^2 \les \dn, \quad \forall\, k\leq 2.
\end{align*}
Combining the above two estimates, we can derive
\begin{align*}
\int_{\Hb_v^{u}}|\Lb\big(\cL_{\Omega_{ij}}^k r^{1+\f12 \ga_0}(\tilde{\rho}, \tilde\si)\big)|^2 dud\om&\les \int_{\Hb_v^{u}}r^{\ga_0}|\Lb\cL_{\Omega_{ij}}^k(\tilde{\rho}, \tilde{\si})|^2 +r^{\ga_0-2}|\cL_{\Omega_{ij}}^k(\tilde{\rho}, \tilde\si)|^2 \\
&\les \dn u_+^{-2},\quad \forall\, k\leq 1.
\end{align*}
We now apply Lemma \ref{lem:globalSobolev} (1) to $f=(\tilde{\rho}, \tilde\si)$, $\ga=\ga_2=1+\f12 \ga_0$ on the incoming null hypersurface $\Hb_v^{u}$.  In this case, we replace $L$ derivative   by $\Lb$ derivative.
The integral on the sphere $S(-v, v)$ can be bounded  in view of (\ref{12.26.2.17}). As we assumed that $\mathcal{E}_{2, \ga_0}\leq \dn$, we then derive from Lemma \ref{lem:globalSobolev} that
\begin{align*}
r^{2+\ga_0}(|\tilde{\rho}|^2+|\si|^2)\les \dn u_+^{-1}.
\end{align*}
Finally for $\tilde\ab$, we make use of the energy flux through the incoming null hypersurface. In Lemma \ref{lem:globalSobolev} (3), let $\ga=1$, $\ga_2=1$, $H=\tilde\ab$ on the incoming null hypersurface $\Hb_v^{u}$. In view of (\ref{12.25.5.17}), we can bound that
\begin{align*}
\int_{\Hb_v^{u}}|{\sl\cL}_{\Om_{ij}}^k \tilde\ab|^2  \les E[\cL_{\Om_{ij}}^k \tilde F](\Hb_v^{u})\les \dn u_+^{-\ga_0},\quad \forall\, k\leq 2.
\end{align*}
In view of (\ref{12.25.4.17}) and (\ref{12.25.5.17}), we can also show that for $l\le 1$,
\begin{align*}
\int_{\Hb_v^u}|{\sl\nabla}_{\Lb}{\sl\cL}_{\Om_{ij}}^l (r\tilde\ab)|^2 dvd\om &\les \int_{\Hb_v^u}|{\sl\nabla}_{\Lb} {\sl\Lie}_{\Om_{ij}}^l \tilde \ab|^2+r^{-2}|{\sl\Lie}_{\Om_{ij}}^l \tilde\ab|^2 \\
&\les \int_{\Hb_v^u}|\ab[\Lie_\pa \Lie_{\Om_{ij}}^l \tilde{F}]|^2+r^{-2}|\ab[\Lie_{\Om_{ij}}^l \tilde F]|^2 \\
&\les \dn u_+^{-2-\ga_0}.
\end{align*}
The integral on $S(-v, v)$ is controlled by (\ref{12.26.2.17}). Thus by using  Lemma \ref{lem:globalSobolev}, we derive that
\begin{align*}
r^2|\ab|^2\les \dn u_+^{-1-\ga_0}.
\end{align*}
\end{proof}

As a corollary of  Proposition \ref{prop:decay:pt:Max}, we derive below a result which will be crucial in the last subsection.
\begin{cor}
Let $\I_Y[F]=r (|F_{LY}|+|F_{AY}) +u_+ |F_{\Lb Y}|$ for $Y\in \Ga$. In the exterior region $\{t+R\leq r, v\le v_*\}$,  there holds
\begin{equation}\label{1.23.1.17}
\I_Y^2 [F]\les (q_0^2+\dn +\dn r^{2-\ga_0}u_+^{-1}) u_+^{2\zeta(Y)}.
\end{equation}
\end{cor}
\begin{proof}
We first can check that  if $Y=\pa$,  $\I_Y [F]\les r|F|$, and if $Y=\Omega_{ij}, \Omega_{0i}$
\begin{equation*}
\I_Y[F]\les r^2(|\a|+|\sigma|)+ru_+(|\ab|+|\rho|).
\end{equation*}
Thus in view of Proposition \ref{prop:decay:pt:Max},  if $ \zeta(Y)=0,\, Y\in \Ga$,
\begin{equation*}
\I_Y[F]\les \dn^\f12 r^{1-\f12 \ga_0}u_+^{-\f12} +\frac{u_+}{r}q_0,
\end{equation*}
and if $\zeta(Y)=-1, Y\in \Ga$,
\begin{equation*}
\I_Y[F]\les r^{-1}q_0+\dn^\f12 u_+^{-\frac{1+\ga_0}{2}}\les (\dn^\f12+q_0)u_+^{-1}.
\end{equation*}
The result then follows by combining the above estimates.
\end{proof}

\subsection{Energy decay estimates for the Maxwell field}
In this subsection, we obtain energy and weighted energy estimates for the Maxwell field.
We prove the following result.
\begin{prop}
\label{prop:BT:imp:Max}
Under the bootstrap assumptions \eqref{eq:BT:energy}-\eqref{eq:BT:PWE:Max}, the energy fluxes for the Maxwell field verify the following decay estimates
\begin{align}
\label{eq:BT:imp:Max:energy}
&E[\cL_Z^k \tilde{F}](\H_{u_1}^{-u_2})+E[\cL_Z^k \tilde{F}](\Hb_{-u_2}^{u_1})\les (\mathcal{E}_{k, \ga_0}+\dn^2) (u_1)_+^{-\ga_0+2\zeta(Z^k)}
\end{align}
as well as the $r$-weighted energy decay estimates
\begin{align}
\notag
 & \iint_{\mathcal{D}_{u_1}^{- u_2}}r^{\ga_0-1} (|\a[\cL_Z^k\tilde F]|^2+|\rho[\cL_Z^k\tilde{F}]|^2+|\sigma[\cL_Z^k\tilde F]|^2)+\int_{\H_{u_1}^{-u_2}}r^{\ga_0}|\a[\cL_Z^k\tilde F]|^2 \\
 \label{eq:BT:imp:Max:pwe}
 &+\int_{\Hb_{-u_2}^{u_1} }r^{\ga_0}(|\rho[\cL_Z^k\tilde{F}]|^2+|\si[\cL_Z^k\tilde F]|^2)\les (\mathcal{E}_{k, \ga_0}+\dn^2) (u_1)_+^{2\zeta(Z^k)}
\end{align}
for all  $-v_*\le u_2<u_1\le-\frac{R}{2}$,   $k\leq 2$, $Z^k\in \Gamma^k$.
\end{prop}
By choosing $\mathcal{E}_{2, \ga_0}$ and $\dn$ suitably small, we can then  improve the bootstrap assumption \eqref{eq:BT:PWE:Max} as well as the bootstrap
 assumption   \eqref{eq:BT:energy} for the Maxwell  field.   The proof for the above proposition is based on the following estimates on  $J[\phi]=\Im(\phi \cdot\overline{D\phi})$.
\begin{prop}
\label{prop:est4J}
Under the bootstrap assumptions \eqref{eq:BT:energy}-\eqref{eq:BT:PWE:Max}, we can bound the term $J[\phi]$ as follows:
\begin{equation}
\label{eq:J:est:need:all}
\begin{split}
\iint_{\mathcal{D}_{u_1}^{-u_2}} r^{\ga_0}u_+^{1+\ep}|\cL_Z^k \J|^2+r^{\ga_0+1}|\cL_Z^k J_{L}|^2 +u_+^{1+\ga_0+\ep}|\cL_Z^k J_{\Lb}|^2 \les \dn^2 (u_1)_+^{2\zeta(Z^k)}
\end{split}
\end{equation}
for all $-v_*\le u_2<u_1\le-\frac{R}{2}$,  $k\leq 2$, $Z^k\in \Gamma^k$, where $\cL_Z^k \J$, $\cL_Z^k J_L$ and $\cL_Z^k J_{\Lb}$ represent the angular, $L$ and $\Lb$ components of the one form $\cL_Z^k J[\phi]$ respectively.
\end{prop}
The rest of this subsection is devoted to proving the above two propositions. We first show that Proposition \ref{prop:est4J} implies Proposition \ref{prop:BT:imp:Max}. To obtain the energy estimates \eqref{eq:BT:imp:Max:energy}, we apply the energy identity \eqref{eq:EEst:id:out} to $\phi=0$ and $\G=\cL_Z^k \tilde{F}$.
% The latter $\tilde{F}$ is the chargeless part of the solution $F$ of the original massive MKG equation.
 By the assumption on the initial data, we derive that
\begin{align*}
E[\cL_Z^k \tilde{F}](\H_{u_1}^{-u_2})+E[\cL_Z^k \tilde{F}](\Hb_{-u_2}^{u_1})\les \mathcal{E}_{k, \ga_0} (u_1)_+^{-\ga_0+2\zeta(Z^k)}+
\iint_{\mathcal{D}_{u_1}^{-u_2}} |(\cL_Z^k \tilde{F})_{0\ga} (\cL_Z^k J[\phi])^\ga |.
\end{align*}
The  error term can be bounded with the help of Cauchy-schwarz inequality,
\begin{align*}
|(\cL_Z^k \tilde{F})_{0\ga}& (\cL_Z^k J[\phi])^\ga|\les |(\cL_Z^k \tilde{F})_{L\ga} (\cL_Z^k J[\phi])^\ga|+|(\cL_Z^k \tilde{F})_{\Lb\ga} (\cL_Z^k J[\phi])^\ga|\\
&\les u_+^{-1-\ep}(|\a[\cL_Z^k \tilde{F}]|^2+|\rho[\cL_Z^k \tilde{F}]|^2)+u_+^{1+\ep}|\cL_Z^k J[\phi]_{\Lb}|^2+ r^{-1-\ep}|\ab[\cL_{Z}^k\tilde F] |^2+r^{1+\ep}|\cL_Z^k \J [\phi]|^2.
\end{align*}
 We then employ (\ref{eq:J:est:need:all}) to treat the nonlinear term involving $J[\phi]$, and apply Lemma \ref{1.6.4.18} to derive the estimate \eqref{eq:BT:imp:Max:energy}.

Next for the $r$-weighted energy estimates \eqref{eq:BT:imp:Max:pwe}, we apply the identity \eqref{pWescaout} to $\phi=0$, $\G=\cL_Z^k \tilde{F}$ and  $p=\ga_0<2$ with the help of the second identity in Lemma \ref{lem:commutator} and (\ref{1.6.5.18}). Then the right hand side can be bounded by the initial data and all the other terms possess positive signs except the nonlinear term $(\cL_Z^k\tilde{F})_{X \ga} (\cL_Z^k J[\phi])^\ga$. We therefore can derive that
 \begin{align*}
  & \iint_{\mathcal{D}_{u_1}^{- u_2}}r^{\ga_0-1} (|\a[\cL_Z^k\tilde F]|^2+|\rho[\cL_Z^k\tilde{F}]|^2+|\si[\cL_Z^k\tilde F]|^2) \\
&+\int_{\H_{u_1}^{-u_2}}r^{\ga_0}|\a[\cL_Z^k\tilde F]|^2 +\int_{\Hb_{-u_2}^{u_1} }r^{\ga_0}(|\rho[\cL_Z^k\tilde F]|^2+|\si[\cL_Z^k\tilde F]|^2)\\
&\les \mathcal{E}_{k, \ga_0} (u_1)_+^{2\zeta(Z^k)}+\iint_{\mathcal{D}_{u_1}^{-u_2}} |\cL_Z^k \tilde{F}_{X\ga}(\cL_Z^k J[\phi])^\ga |.
\end{align*}
For the last integral, with $J=J[\phi]$, we first can estimate that
\begin{align*}
|(\cL_Z^k\tilde{F})_{X \ga} (\cL_Z^k J)^\ga|&\les r^{\ga_0}(|\rho[\cL_Z^k \tilde F]||\cL_Z^k J_{L}|+|\a[\cL_Z^k \tilde F]||\cL_Z^k \J|)\\
&\les r^{\ga_0}u_+^{-1-\ep}|\a[\cL_Z^k \tilde F]|^2+r^{\ga_0}u_+^{1+\ep}|\cL_Z^k \J|^2+\ep_1 r^{\ga_0-1}|\rho[\cL_Z^k \tilde{F}]|^2+\ep_1^{-1}r^{\ga_0+1}|\cL_Z^k J_{L}|^2
\end{align*}
for all $\ep_1>0$. The integral of the first term can be absorbed by using Gronwall's inequality and the third term can be absorbed for sufficiently small $\ep_1$ which depends only on $\ga_0$ and the implicit universal constant. Once we have chosen $\ep_1$, the rest two terms involving the terms of $J$ can be bounded by using Proposition \ref{prop:est4J}. Next, we prove Proposition \ref{prop:est4J}.
% In particular, to prove Proposition \ref{prop:BT:imp:Max}, it remains to bound the nonlinearity $J$.

%\textsl{Proof for Proposition \ref{prop:est4J}}:
\begin{proof}[Proof for Proposition \ref{prop:est4J}]
We first consider the case when $k=0$. In this case, we will show
\begin{equation}\label{1.15.1.18}
\iint_{\mathcal{D}_{u_1}^{-u_2}} r^{\ga_0}u_+^{1+\ep}| \J|^2+r^{\ga_0+1}| J_{L}|^2 +u_+^{1+\ga_0+\ep}|J_{\Lb}|^2\les \dn^2 (u_1)_+^{-1-\ga_0+\ep}.
\end{equation}
Note that  in view of (\ref{eq:decay:scal:pt:phi}) and (\ref{eq:EEst:decay:0}), we can derive
\begin{align*}
\iint_{\mathcal{D}_{u_1}^{-u_2}} &|\phi|^2(r^{\ga_0}u_+^{1+\ep}| \D \phi|^2+r^{\ga_0+1}| D_L \phi|^2 +u_+^{1+\ga_0+\ep}|D_{\Lb} \phi|^2)\\
&\les \dn (u_1)_+^{-2+\ep} \iint_{\mathcal{D}_{u_1}^{-u_2}} |\D \phi|^2+ |D_L \phi|^2 + r^{-1-\ep }u_+^{1+\ep}|D_{\Lb}\phi|^2 \\
&\les \dn\E_{0,\ga_0}  (u_1)_+^{-1-\ga_0+\ep}.
\end{align*}
In view of the definition of $J[\phi]$, this gives (\ref{1.15.1.18}).

Next, we consider the cases when $k=1,2$. By the definition of Lie derivatives, for any $X\in \Gamma$, we can show that
\begin{equation}\label{12.07.2.17}
\begin{split}
\cL_X J_{\mu}[\phi]&= X \Im(\phi\cdot \overline{D_\mu \phi})-\Im (\phi\cdot \overline{D_{[X, \pa_\mu]}\phi})\\
&=\Im(D_X\phi \cdot \overline{D_\mu \phi} )+\Im(\phi\cdot \overline{D_\mu D_X \phi})+\Im(\phi\cdot \overline{i F_{X \mu}}\phi)\\
&=\Im(D_X\phi \cdot \overline{D_\mu \phi} )+\Im(\phi\cdot \overline{D_\mu D_X \phi})- F_{X \mu}|\phi|^2.
\end{split}
\end{equation}
In particular we can bound that
\begin{align*}
|\cL_X J_{\mu}[\phi]|\les |D_X\phi||D_\mu \phi|+|\phi||D_\mu D_X\phi|+|F_{X \mu}||\phi|^2.
\end{align*}
For the second order derivative, we can show that
\begin{align*}
\cL_X \cL_Y J_\mu[\phi]&=X(\cL_Y J_\mu[\phi])-(\cL_Y J[\phi])([X, \pa_\mu])\\
&=X\left(\Im(D_Y\phi \cdot \overline{D_\mu \phi} )+\Im(\phi\cdot \overline{D_\mu D_Y \phi})- F_{Y \mu}|\phi|^2\right)\\
&-\Im(D_Y\phi \cdot \overline{D_{[X, \pa_\mu]} \phi} )-\Im(\phi\cdot \overline{D_{[X, \pa_\mu]} D_Y \phi})+ F_{Y [X, \pa_\mu]}|\phi|^2\\
&=\Im(D_X D_Y\phi\cdot \overline{D_\mu\phi})+\Im(D_Y\phi\cdot (\overline{ D_\mu D_X\phi+i F_{X\pa_\mu}\phi}))\\
&+\Im(D_X\phi\cdot \overline{D_\mu D_Y \phi})+\Im(\phi\cdot (\overline{D_\mu D_X D_Y \phi+iF_{X\mu} D_Y\phi}))\\
&-(\cL_X F)_{Y\mu }|\phi|^2-F_{[X, Y]\mu}|\phi|^2-F_{Y\mu} D_X|\phi|^2.
\end{align*}
Therefore we can estimate that
\begin{align*}
|\cL_X \cL_Y J_\mu[\phi]|&\les |D_X D_Y\phi||D_\mu\phi|+|D_Y\phi| |D_\mu D_X\phi|+|F_{X\mu}||\phi||D_Y\phi|
+|D_X\phi||D_{\mu} D_Y \phi|\\&
+|\phi||D_\mu D_X D_Y \phi|+|(\cL_X F)_{Y\mu }||\phi|^2+|F_{[X, Y]\mu}||\phi|^2+|F_{Y\mu}| |\phi||D_X\phi|.
\end{align*}
Let's only consider the second order derivatives as the one derivative case should be easier and the corresponding estimate can follow in a similar way. For the first term on the right hand side of the above inequality, we bound $|D_\mu\phi|$ by the pointwise estimates obtained in Proposition \ref{prop:decay:scal} and the other term $|D_XD_Y\phi|$ by the bound of  energy flux along the outgoing null hypersurface in (\ref{eq:BT:energy}). Indeed from Proposition \ref{prop:decay:scal}, we first can bound that
\begin{align*}
r^{\ga_0}u_+^{1+\ep}|\D\phi|^2+r^{\ga_0+1}|D_L\phi|^2 +u_+^{1+\ga_0+\ep}|D_{\Lb}\phi|^2\les \dn u_+^{-3+2\ep}.
\end{align*}
Therefore we can show that
\begin{align*}
\iint_{\mathcal{D}_{u_1}^{-u_2}} |D_XD_Y\phi|^2(r^{\ga_0}u_+^{1+\ep}|\D\phi|^2+r^{\ga_0+1}|D_L\phi|^2 +u_+^{1+\ga_0+\ep}|D_{\Lb}\phi|^2)\les \dn^2 (u_1)_+^{2\zeta(XY)-2-\ga_0+2\ep}.
\end{align*}
Next for $|D_X\phi||D_\mu D_Y\phi|$ or $|D_Y\phi||D_\mu D_X\phi|$, it suffices to consider one of them due to symmetry. From Proposition \ref{prop:decay:scal}, we have
\[
|D_X\phi|^2\les \dn r^{-\frac{5}{2}+\ep}u_+^{\f12 -\ga_0 +2\zeta(X)}.
\]
Therefore we can bound that
\begin{align*}
&\iint_{\mathcal{D}_{u_1}^{-u_2}} |D_X\phi|^2(r^{\ga_0}u_+^{1+\ep}|\D D_Y\phi|^2+r^{\ga_0+1}|D_L D_Y\phi|^2 +u_+^{1+\ga_0+\ep}|D_{\Lb}D_Y\phi|^2)\\
&\les \dn  (u_1)_+^{2\zeta(X)}\iint_{\mathcal{D}_{u_1}^{-u_2}} u_+^{-1+2\ep}|\D D_Y\phi|^2+r u_+^{-2+\ep}|D_L D_Y\phi|^2 + r^{-1-\ep }u_+^{3\ep}|D_{\Lb}D_Y\phi|^2 \\
&\les \dn^2  (u_1)_+^{2\zeta(XY)-\ga_0+2\ep}.
\end{align*}
Next we estimate $|\phi||D_\mu D_X D_Y\phi|$. For this term, we can bound $\phi$ by the $L^\infty$ norm in (\ref{eq:decay:scal:pt:phi}) and the other one by the bound in (\ref{eq:BT:energy}) for the energy fluxes through incoming and outgoing null hypersurfaces. More precisely we have
\begin{align*}
&\iint_{\mathcal{D}_{u_1}^{-u_2}} |\phi|^2(r^{\ga_0}u_+^{1+\ep}|\D D_X D_Y\phi|^2+r^{\ga_0+1}|D_L D_X D_Y\phi|^2 +u_+^{1+\ga_0+\ep}|D_{\Lb}D_XD_Y\phi|^2)\\
&\les \dn (u_1)_+^{-2+\ep} \iint_{\mathcal{D}_{u_1}^{-u_2}} |\D D_X D_Y\phi|^2+ |D_L D_X D_Y\phi|^2 + r^{-1-\ep }u_+^{1+\ep}|D_{\Lb}D_X D_Y\phi|^2 \\
&\les \dn^2  (u_1)_+^{2\zeta(XY)-1-\ga_0+\ep}.
\end{align*}
The remaining ones are cubic nonlinear terms. By symmetry, it suffices to consider $|F_{X \mu }||\phi||D_Y\phi|$ and $(|(\cL_X F)_{Y\mu}|+|F_{[X, Y]\mu}|)|\phi|^2$. In view of
(\ref{12.7.3.17}),  we can roughly bound such terms
as follows,
\begin{equation*}
|F_{X\mu}|^2 \les (\dn+q_0^2) u_+^{-2+2\zeta(X)},\quad |F_{[X, Y]\mu}|^2\les (\dn+q_0^2) u_+^{2\zeta(XY)-2}.
\end{equation*}

Therefore we have
\begin{align*}
\iint_{\mathcal{D}_{u_1}^{-u_2}}& r^{\ga_0+1+\ep}(|F_{X\mu}|^2|\phi|^2|D_Y\phi|^2+|F_{[X, Y]\mu}|^2|\phi|^4)\\
&\les \dn(\dn+q_0^2) (u_1)_+^{-4+\ep+2\zeta(X)} \iint_{\mathcal{D}_{u_1}^{-u_2}} |D_Y\phi|^2+u_+^{2\zeta(Y)}|\phi|^2 \\
&\les \dn^2(\dn+q_0^2)  (u_1)_+^{2\zeta(XY)-3-\ga_0+\ep}.
\end{align*}
For the last one $|(\cL_X F)_{Y\mu}| |\phi|^2$, we need the following facts
\begin{equation}\label{12.13.1.17}
|\Lie_X F-\Lie_X \tilde F|\les r^{\zeta(X)-2}|q_0|,\quad \quad |\G_{X\nu}|\les r^{\zeta(X)+1}|\G_{\mu\nu }|\mbox{ where } X\in \Ga.
\end{equation}
 Here $\G$ is a two form.
We can bound $\phi$ by the pointwise bound in (\ref{eq:decay:scal:pt:phi}) and $(\cL_X \tilde F)_{Y\mu}$ by the bound of energy flux in (\ref{eq:BT:energy}). We thus can derive that
\begin{align*}
&\iint_{\mathcal{D}_{u_1}^{-u_2}} |\phi|^4(r^{\ga_0}u_+^{1+\ep}|(\cL_X F)_{Y e_A}|^2+r^{\ga_0+1}|(\cL_X F)_{Y L}|^2 +u_+^{1+\ga_0+\ep}|(\cL_X F)_{Y \Lb}|^2)\\
&\les \dn^2  \iint_{\mathcal{D}_{u_1}^{-u_2}}  r^{-5+\ga_0+\ep}u_+^{-2\ga_0}|(\cL_X F)_{Y \mu}|^2 \\
&\les \dn^2  \iint_{\mathcal{D}_{u_1}^{-u_2}}  r^{-3+\ga_0+\ep}u_+^{-2\ga_0+2\zeta(Y)}|\cL_X F|^2 \\
&\les \dn^2(\dn+q_0^2)  (u_1)_+^{2\zeta(XY)-2-2\ga_0+\ep}.
\end{align*}
 For deriving the last inequality, we used (\ref{12.13.1.17}) to decompose $\Lie_X F$. With the help of the fact that $-3+\ga_0+\ep<-1$, which can be seen by our assumption of $\ga_0$, we then  estimated the component $\ab[\cL_X \tilde F]$  by using the energy flux through the incoming null hypersurfaces and other components of $\cL_X \tilde F$ by using energy flux on the outgoing null hypersurfaces.
Combining all the above estimates, we can derive estimate \eqref{eq:J:est:need:all}. This finished the proof for Proposition \ref{prop:est4J}.
\end{proof}

\subsection{ Decay of Energy  estimates for the scalar field}\label{Sec1.15}
We derive the energy decay estimates for the scalar field in this section under the bootstrap assumptions \eqref{eq:BT:energy}-\eqref{eq:BT:PWE:Max}. We show that
\begin{prop}
\label{prop:BT:imp:scal}
Under the bootstrap assumptions \eqref{eq:BT:energy}-\eqref{eq:BT:PWE:Max}, we have the energy flux decay for the scalar field
\begin{align}
\label{eq:BT:imp:energy:scal}
E[D_Z^k\phi](\H_{u_1}^{-u_2})+E[D_Z^k \phi](\Hb_{-u_2}^{u_1})\les (\mathcal{E}_{k, \ga_0}+\dn^2) (u_1)_+^{-\ga_0+2\zeta(Z^k)}
\end{align}
and the $r$-weighted energy flux decay estimate
\begin{align}
\label{eq:BT:imp:pwe:scal}
\int_{\H_{u_1}^{-u_2}}r |D_LD_Z^k\phi|^2+\int_{\Hb_{-u_2}^{u_1} }r(|\D D_Z^k\phi|^2+|D_Z^k\phi|^2) \les (\mathcal{E}_{k, \ga_0} +\dn^2) (u_1)_+^{1-\ga_0+2\zeta(Z^k)}
\end{align}
for all $-v_*\le u_2<u_1\le-\frac{R}{2}$, $k\leq 2$, $Z^k\in \Gamma^k$.
\end{prop}
To show the above proposition, we need  the following estimates on the commutators.
\begin{lem}
\label{prop:Q:need:r:all}
For  $1\le k\leq 2$, let $Z^k=Z_1\cdots Z_k \in\Gamma^k$. Under the bootstrap assumptions \eqref{eq:BT:energy}-\eqref{eq:BT:PWE:Max},  when $k=2$, we can bound
\begin{equation} \label{eq:Q:need:r:all}
\begin{split}
\iint_{\mathcal{D}_{u_1}^{-u_2}} r^{1+\ep}u_+^{1+\ep}|(\Box_A-1)D_Z^k\phi|^2 &\les (\E_{0,\ga_0}+ \dn^2) (u_1)_+^{2\zeta(Z^k)+1-\ga_0}+q_0^2\iint_{\mathcal{D}_{u_1}^{-u_2}}{u_+}^{2\zeta(Z^k)+2\ep} |D \phi|^2\\
&+q_0^2 \iint_{\mathcal{D}_{u_1}^{-u_2}} {u_+}^{2(\zeta(Z^k)+2\ep}\sum_{Z\in\{ Z_1, Z_2\}}u_+^{-2\zeta(Z)} |D D_Z \phi|^2,
\end{split}
\end{equation}
where $-v_*\le u_2<u_1\le -\f12 R$ in the above estimate.
When $k=1$, the same estimate holds with the terms in the second line vanished.
\end{lem}

We can actually prove the improved estimate
\begin{prop}
\label{1.27.3.18}
 Under the bootstrap assumptions \eqref{eq:BT:energy}-\eqref{eq:BT:PWE:Max},  we have for $Z^k\in \Ga^k$ with $k\le 2$ \begin{footnote}{ Note that when $k=0$,  as $\phi$ verifies the massive MKG equation, we have $\Box_A\phi-\phi=0$. Thus (\ref{1.22.1.17})  automatically holds.}\end{footnote}
 \begin{equation}\label{1.22.1.17}
\iint_{\mathcal{D}_{u_1}^{-u_2}} r^{1+\ep}u_+^{1+\ep}|(\Box_A-1)D_Z^k\phi|^2 \les (\E_{k,\ga_0}+ \dn^2)(u_1)_+^{2\zeta(Z^k)+1-\ga_0},
\end{equation}
where $-v_*\le u_2<u_1\le -\f12 R$.
\end{prop}
 Proposition \ref{1.27.3.18} has to be proved together with (\ref{eq:BT:imp:energy:scal}) in Proposition \ref{prop:BT:imp:scal} by using  Lemma \ref{prop:Q:need:r:all}.  We now prove Lemma \ref{prop:Q:need:r:all} first, then use it to prove Proposition \ref{1.27.3.18} and Proposition \ref{prop:BT:imp:scal}.
%\textsl{Proof for Proposition \ref{prop:Q:need:r:all}}:
\begin{proof}[Proof of Lemma \ref{prop:Q:need:r:all}] We start with considering  $(\Box_A-1)D_Z^k\phi$.
 When $k=1$, in view of Lemma \ref{lem:commutator} and the definition in (\ref{12.9.1.17}), we have
\[
(\Box_A-1)D_Z\phi=Q(F, \phi, Z).
\]
 When $k=2$, let $Z^2=XY$ with $X$, $Y\in \Gamma$. Lemma \ref{lem:Est4commu:id:2} and Lemma \ref{cor1} imply that
 \begin{equation}\label{12.7.6.17}
 \begin{split}
 (\Box_A-1)D_X D_Y\phi&=Q(F, D_X\phi, Y)+Q(F, D_Y\phi, X)+Q(F, \phi, \Lie_X Y)\\
 &+Q(\cL_X F, \phi, Y)-2F_{X\mu}F^{\mu}_{\, \, Y}\phi.
 \end{split}
 \end{equation}
Note that the commutators consist of two parts: one has the same structure as $Q(F, \phi, Z)$, which is  quadratic in $F$ and $D\phi$ and will be referred to as the quadratic part; the other one is a set of the cubic terms. Let's first estimate the cubic terms for which we can use Lemma \ref{lem:Est4commu:2}. Indeed we can show that
\begin{align}
&\iint_{\mathcal{D}_{u_1}^{-u_2}} r^{1+\ep}u_+^{1+\ep}|F_{X\mu}F^{\mu}_{\, \, Y}|^2|\phi|^2 \nn\\
&\les \iint_{\mathcal{D}_{u_1}^{-u_2}} r^{1+\ep}u_+^{1+\ep+2\zeta(XY)}\big(u_+^{2}|\ab|^2+ r^{2}(|\si|^2+|\a|^2+|\rho|^2+|\a||\ab|)\big)^2|\phi|^2 \nn\\
&\les \iint_{\mathcal{D}_{u_1}^{-u_2}} (\dn+q_0^2)^2 r^{-\ga_0+1+\ep}u_+^{-3+\ga_0+\ep+2\zeta(XY)}|\phi|^2 \nn\\
&\les (\dn+q_0^2)^2\E_{0,\ga_0} (u_1)_+^{2\zeta(XY)-1+2\ep -\ga_0},\label{12.7.5.17}
\end{align}
where we also employed Proposition \ref{prop:decay:pt:Max} and (\ref{eq:EEst:decay:0}).

Next we will control the main part of the commutator, which takes the quadratic form $Q(F, f, X)$.
%Notice that
%for any $X$, $Y\in \Gamma$, the commutator $[X, Y]$ also belongs to $\Gamma$.  With $0\le l_1+l_2\leq k-1$, $Z^1=Y$. We show that
%\begin{equation}
%\label{eq:Q:bd}
%\begin{split}
%\iint_{\mathcal{D}_{u_1}^{-u_2}} r^{1+\ep}u_+^{1+\ep}|Q(\cL_{Z}^{l_1}F, D_Z^{l_2}\phi, X)|^2& \les (\dn^2+\E_{0,\ga_0}) (u_1)_+^{2\zeta(X Z^{l_1}Z^{l_2})+1 %-\ga_0}\\
%&+q_0^2  \iint_{\mathcal{D}_{u_1}^{-u_2}} {u_+}^{2\zeta(X Z^{l_1} Z^{l_2})+2\ep} \cdot u_+^{-2\zeta(Z^{l_2})}|D D_Z^{l_2} \phi|^2.
%\end{split}
%\end{equation}
%As $l_1+l_2\leq k-1$,
 We first apply Lemma \ref{lem:Est4commu:1} to scalar field $D_{Z}^{l}\phi$ with  $\G=F$ for all $l\leq k-1$.
 Denote $\phi_{,l}=D_Z^l\phi$ and  $Z^1\in \Ga$.  We can employ Proposition \ref{prop:decay:scal}, Proposition \ref{prop:decay:pt:Max}, (\ref{12.7.3.17}) and the fact that $\dn\le 1$ to  bound that
\begin{align*}
|Q(F, \phi_{,l}, X)|^2&\les r^{2\zeta(X)+2}(|\a|^2|D \phi_{,l}|^2+|\ud F|^2|D_L \phi_{,l}|^2+|\si|^2|\D \phi_{,l}|^2)\\
&+u_+^{2\zeta(X)+2}(|\rho|^2|D_{\Lb}\phi_{,l}|^2
+|\ab|^2|\D \phi_{,l}|^2)\\
&+\left(u_+^{2\zeta(X)+2}|J_{\Lb}[\phi_{,l}]|^2+r^{2\zeta(X)+2}(|J_L [\phi_{,l}]|^2 +|\J [\phi_{,l}]|^2)+r^{2\zeta(X)}| F|^2\right)|\phi_{,l}|^2\\
&\les  u_+^{2\zeta(X)}r^{-1-\ep}\left((q_0^2+\dn)u_+^{-1+\ep}|D\phi_{,l}|^2+\dn(r^{1+\ep}u_+^{-2}|D_L \phi_{,l}|^2+u_+^{-3+\ep}|\phi_{,l}|^2)\right)\\
&+q_0^2r^{2\zeta(X)-2}|D_L \phi_{,l}|^2 
\end{align*}
where $\ud F$ represents all the components of $F$ except $\a[F]$. 

Now from the bootstrap assumption (\ref{eq:BT:energy}) and the fact that $l\leq k-1$, we can show that
\begin{align*}
\int_{\H_u^{-u_2}}|D D_Z^l\phi|^2  \les \dn u_+^{-\ga_0+2\zeta( Z^l)+2\zeta(\pa)}=\dn u_+^{-\ga_0-2+2\zeta( Z^l)}.
\end{align*}
By making use of the Lorentz boost, we have the improved energy flux decay
\begin{align*}
\int_{\H_u^{-u_2}}r^2|D_L D_Z^l\phi|^2 &\les \int_{\H_u^{-u_2}}(|D_{\Om_{0j}} D_Z^l\phi|^2+|D_{\Om_{ij}}D_Z^l\phi|^2 +u_+^2|D D_Z^l\phi|^2) \\
&\les \dn u_+^{-\ga_0+2\zeta(\Om Z^l)}=\dn u_+^{-\ga_0+2\zeta(Z^l)}.
\end{align*}
Therefore by using the above estimate and (\ref{eq:BT:energy}) we can obtain
\begin{equation}\label{12.7.4.17}
\begin{split}
\iint_{\mathcal{D}_{u_1}^{-u_2}} r^{1+\ep}u_+^{1+\ep}|Q(F, D_Z^{l}\phi, X)|^2& \les \dn^2(u_1)_+^{2\zeta(X Z^l)-1-\ga_0+2\ep}\\
&+q_0^2  \iint_{\mathcal{D}_{u_1}^{-u_2}} u_+^{2\zeta(XZ^l)+2\ep-2\zeta(Z^l)} |D D_Z^l \phi|^2
\end{split}
\end{equation}
for  $l\leq k-1$ and $k\le 2$. This in particular implies the $k=1$ case in  Lemma \ref{prop:Q:need:r:all} if we take $l=0$.

We will estimate the term $Q(\cL_Y F, \phi, X)$ by applying Lemma \ref{lem:Est4commu:1} to $\cL_Y F$ and the scalar field $\phi$. To treat the term of $\delta (\cL_Y F)$, we can use the  second identity in Lemma \ref{lem:commutator} and (\ref{EQMKG}).
Also by using Proposition \ref{prop:decay:scal} and (\ref{12.13.1.17}), we then can bound that
\begin{align}
|Q(\cL_Y F, \phi, X)|^2 &\les r^{2\zeta(X)+2}(|\a[\cL_Y F]|^2 |D \phi|^2+|\ud{\cL_Y F}|^2|D_L \phi|^2+|\si[\cL_Y F]|^2|\D \phi|^2)\nn\\
&+u_+^{2\zeta(X)+2}(|\rho[\cL_Y F]|^2|D_{\Lb}\phi|^2
+|\ab[\cL_{Y}F]|^2|\D \phi|^2)\nn\\
&+\left(u_+^{2\zeta(X)+2}|\cL_Y J_{\Lb}|^2+r^{2\zeta(X)+2}(|\cL_Y J_L |^2 +|\cL_Y\J |^2)+r^{2\zeta(X)}|\cL_Y F|^2\right)|\phi|^2\nn\\
&\les \dn r^{2\zeta(X)-\frac{1}{2}+\ep}u_+^{\f12 -\ga_0}(u_+^{-2}|\a[\cL_Y \tilde{F}]|^2 +r^{-2}|\ud{\cL_Y\tilde F}|^2+r^{-2}|\si[\cL_Y \tilde F]|^2)\nn\\
&+\dn u_+^{2\zeta(X)+\frac{5}{2}-\ga_0}r^{-\frac{5}{2}+\ep}(u_+^{-2}|\rho[\Lie_Y  \tilde{F}]|^2+r^{-2}|\ab[\Lie_Y \tilde{F}]|^2)\label{12.07.1.17}\\
&+\left(u_+^{2\zeta(X)+2}|\cL_Y J_{\Lb}|^2+r^{2\zeta(X)+2}(|\cL_Y J_L |^2 +|\cL_Y\J |^2)\right)|\phi|^2\nn\\
&+u_+^{2\zeta(XY)}r^{-2}q_0^2 (|D\phi|^2+r^{-2}|\phi|^2)\nn,
\end{align}
where the terms in (\ref{12.07.1.17}) can be incorporated into those above the line.

Recall from (\ref{12.07.2.17}) that
\begin{align*}
|\cL_Y J_{\mu}[\phi]|\les |D_Y\phi||D_\mu \phi|+|\phi||D_\mu D_Y\phi|+|F_{Y \mu}||\phi|^2.
\end{align*}
By using Proposition \ref{prop:decay:scal}, we can  therefore  bound that
\begin{align*}
&\left(u_+^{2\zeta(X)+2}|\cL_Y J_{\Lb}|^2+r^{2\zeta(X)+2}(|\cL_Y J_L |^2 +|\cL_Y\J |^2)\right)|\phi|^2\\
&\les \dn^2 r^{\ep-\frac{7}{2}}u_+^{2\zeta(X)-2\ga_0-\frac{3}{2}}(|D_Y\phi|^2+u_+^\frac{3}{2} r^{-\f12}|D D_Y\phi|^2)+\dn u_+^{2\zeta(X)}\I^2_Y[F] u_+^{-\ga_0}r^{-3}|\phi|^2,
\end{align*}
where the definition and the treatment of $\I_Y[F]$ can be found in  (\ref{1.23.1.17}).
The above  estimate together with the bootstrap assumptions implies  that
\begin{align*}
&\iint_{\mathcal{D}_{u_1}^{-u_2}} r^{1+\ep}u_+^{1+\ep}|Q(\cL_Y F, \phi, X)|^2 \\
&\les \dn\iint_{\mathcal{D}_{u_1}^{-u_2}}u_+^{\frac{3}{2}+\ep-\ga_0}r^{2\zeta(X)+\f12+2\ep}(u_+^{-2}|\a[\cL_Y \tilde F]|^2 +r^{-2}|\ud{\cL_Y \tilde F}|^2+r^{-2}|\si[\cL_Y \tilde F]|^2)\\
&+\dn \iint_{\mathcal{D}_{u_1}^{-u_2}}r^{2\ep-\frac{5}{2}}u_+^{2\zeta(X)-2\ga_0-\frac{1}{2}+\ep}(|D_Y\phi|^2+u_+^{\frac{3}{2}} r^{-\f12}|D D_Y\phi|^2)+ u_+^{-\ga_0-1+2\ep+2\zeta(XY)}|\phi|^2 \displaybreak[0]\\
&+q_0^2\iint_{\mathcal{D}_{u_1}^{-u_2}} u_+^{2\zeta(XY)+2\ep}(|D\phi|^2+r^{-2} |\phi|^2)\\
&\les (\dn^2+\E_{0,\ga_0}) (u_1)_+^{2\zeta(X Y)+1-2\ga_0+3\ep}+q_0^2\iint_{\mathcal{D}_{u_1}^{-u_2}} u_+^{2\zeta(XY)+2\ep}|D\phi|^2,
\end{align*}
where we used $\ga_0>1+3\ep$ to bound the term of $\ud{\Lie_Y \tilde F}$. We also used Proposition \ref{prop:Energyflux:decay:0} for deriving the last inequality.

Thus  we proved
\begin{equation*}
\iint_{\mathcal{D}_{u_1}^{-u_2}} r^{1+\ep}u_+^{1+\ep}|Q(\cL_X F, \phi, Y)|^2\les  (\dn^2+\E_{0,\ga_0}) (u_1)_+^{2\zeta(Z^k)+1-2\ga_0+3\ep}+q_0^2\iint_{\mathcal{D}_{u_1}^{-u_2}} u_+^{2\zeta(Z^k)+2\ep}|D\phi|^2.
\end{equation*}

  Now we  treat the first three terms on the right of (\ref{12.7.6.17}) by applying (\ref{12.7.4.17}) to $Q(F, D_Z^{l_1}\phi, Z^{l_2})$, with some  $0\le l_1\le  l_2\le 1$. Note that $[X,Y]\in \Ga$ since $X, Y\in \Ga$. With $(Z^{l_1}, Z^{l_2})=(X, Y), (Y, X), (Z^0, [X,Y])$, we derive from (\ref{12.7.4.17}) that
\begin{equation*}
\begin{split}
\iint_{\mathcal{D}_{u_1}^{-u_2}} r^{1+\ep}u_+^{1+\ep}|Q(F, D_Z^{l_1}\phi, Z^{l_2})|^2& \les \dn^2(u_1)_+^{2\zeta(Z^2)-1-\ga_0+2\ep}+q_0^2\iint_{\mathcal{D}_{u_1}^{-u_2}}{u_+}^{2\zeta(Z^2)+2\ep} |D \phi|^2\\
&+q_0^2 \iint_{\mathcal{D}_{u_1}^{-u_2}}u_+^{2\zeta(Z^2)+2\ep} \sum_{Z\in \{X, Y\}}u_+^{-2\zeta(Z)} |D D_Z \phi|^2,
\end{split}
\end{equation*}
where we used the fact that the sum of signatures of each $(Z^{l_1},Z^{l_2})$  satisfies $\zeta(Z^{l_1})+\zeta(Z^{l_2})=\zeta(XY)$, which is $\zeta(Z^2)$.
Combining the above two estimates with (\ref{12.7.5.17})  implies \eqref{eq:Q:need:r:all}. We thus finished the proof for Lemma \ref{prop:Q:need:r:all}.
\end{proof}

Now we are ready to prove Proposition \ref{prop:BT:imp:scal} and  Proposition \ref{1.27.3.18}.

\begin{proof}[Proof of Proposition \ref{prop:BT:imp:scal} and  Proposition \ref{1.27.3.18}] It suffices to consider the cases that $k=1,2$ for Proposition \ref{prop:BT:imp:scal}, since in Proposition \ref{prop:Energyflux:decay:0}, we  have completed the case $k=0$.

We apply the energy identity \eqref{eq:EEst:id:out} to $(\phi, \G)=(D_Z^k\phi, 0)$, which implies
\begin{align*}
&E[D_Z^k\phi](\H_{u_1}^{-u_2})+E[D_Z^k \phi](\Hb_{-u_2}^{u_1})\\
&\les \mathcal{E}_{k, \ga_0} (u_1)_+^{-\ga_0+2\zeta(Z^k)}+\iint_{\mathcal{D}_{u_1}^{-u_2}} (|F_{0\mu} ||J[D_Z^k \phi]^\mu|+|(\Box_A-1)D_Z^k\phi|| D_0 D_Z^k\phi|).
\end{align*}
We first  bound the nonlinear terms in the integral on the right hand side as
\begin{align*}
&|F_{0\mu} ||J[D_Z^k \phi]^\mu|+|(\Box_A-1)D_Z^k\phi|| D_0 D_Z^k\phi|\\
&\les  r^{-1-\ep}| D D_Z^k\phi|^2+r^{1+\ep}|(\Box_A-1)D_Z^k\phi|^2+r^{1+\ep}|F_{0\mu}|^2|D_Z^k\phi|^2.
\end{align*}
 The integral of the first term can be controlled by the fluxes on $\H_{u_1}^{-u_2}$ and $\Hb_{-u_2}^{u_1}$.  For the last term, due to $\zeta(\partial_t)=-1$, we apply (\ref{12.7.3.17}) to obtain $
|F_{0\mu}|^2\les (\dn+q_0^2) r^{-2}u_+^{-2}.
$
We can bound that
\begin{align*}
\iint_{\mathcal{D}_{u_1}^{-u_2}} r^{1+\ep} |F_{0\mu}|^2 |D_Z^k \phi|^2 \les  \iint_{\mathcal{D}_{u_1}^{-u_2}} (\dn+q_0^2) u_+^{-3+\ep} |D_Z^k \phi|^2.
%\les \dn(\dn+q_0^2) (u_1)_+^{2\zeta(Z^k)-2-\ga_0+\ep}.
\end{align*}
Thus we can obtain
\begin{equation}\label{1.27.2.18}
\begin{split}
E[D_Z^k\phi](\H_{u_1}^{-u_2})&+E[D_Z^k\phi](\Hb_{-u_2}^{u_1})\les (\mathcal{E}_{k, \ga_0}+\dn^2) (u_1)_+^{-\ga_0+2\zeta(Z^k)}\\
&+\int_{u_2}^{u_1}E[D_Z^k\phi](\H_u^{-u_2}) u_+^{-1-\ep}  du+\int_{-u_1}^{-u_2} v^{-1-\ep} E[D_Z^k\phi](\Hb_v^{u_1}) d v\\
&+\iint_{\mathcal{D}_{u_1}^{-u_2}} r^{1+\ep} |(\Box_A-1) D_Z^k \phi|^2.
\end{split}
\end{equation}
The last term will be estimated by using Lemma \ref{prop:Q:need:r:all}.  We consider (\ref{eq:BT:imp:energy:scal}) for the case $k=1$ and $Z^1=\pa$. In view of (\ref{1.27.2.18}), we have
\begin{align*}
E[D_Z\phi](\H_{u_1}^{-u_2})&+E[D_Z\phi](\Hb_{-u_2}^{u_1})\les (\mathcal{E}_{1, \ga_0}+\dn^2) (u_1)_+^{-\ga_0-2}\\
&+\int_{u_2}^{u_1}E[D\phi](\H_u^{-u_2})( u_+^{-2+2\ep} +u_+^{-1-\ep}) du+\int_{-u_1}^{-u_2} v^{-1-\ep} E[D_Z\phi](\Hb_v^{u_1}) d v.
\end{align*}
By taking all $Z\in \{\pa\}$ and using Lemma \ref{1.6.4.18}, we can obtain
\begin{equation}\label{1.27.1.18}
E[D\phi](\H_{u_1}^{-u_2})+E[D\phi](\Hb_{-u_2}^{u_1})\les (\mathcal{E}_{1, \ga_0}+\dn^2) (u_1)_+^{-\ga_0-2}.
\end{equation}
We can substitute (\ref{1.27.1.18}) to (\ref{eq:Q:need:r:all}) to control the estimate of $D\phi$.  This also completed the case $k=1$ of Proposition \ref{1.27.3.18}.

 Let us  consider the case $k=2$ and $Z^2=\pa Z$ with $Z\in \Ga$.
 Due to $\zeta(\pa Z)=-1+\zeta(Z)$ , by using (\ref{eq:Q:need:r:all}) and (\ref{1.27.2.18}) we can derive
\begin{equation}\label{1.27.4.18}
\begin{split}
&E[D D_Z\phi](\H_{u_1}^{-u_2})+E[DD_Z\phi](\Hb_{-u_2}^{u_1})\les (\mathcal{E}_{2, \ga_0}+\dn^2) (u_1)_+^{-\ga_0+2\zeta(\pa Z)}\\
&+\int_{u_2}^{u_1}u_+^{2\zeta(\pa Z)+2\ep} (u_+^2 |DD_{\pa} \phi|^2+u_+^{-2\zeta(Z)}|D D_Z\phi|^2) du\\
&+\int_{u_2}^{u_1}u_+^{-1-\ep} E[D D_Z \phi](\H_u^{-u_2})du+\int_{-u_1}^{-u_2} v^{-1-\ep} E[D D_Z\phi](\Hb_v^{u_1}) d v .
\end{split}
\end{equation}
By summing over $Z\in \{\pa\}$, applying Lemma \ref{1.6.4.18} gives (\ref{eq:BT:imp:energy:scal}) for all $Z^2=Z_1 Z_2$ with $Z_1, Z_2\in \{\pa\}$.
We then let $Z\in \{\Omega_{\mu \nu}\}$ in (\ref{1.27.4.18}). By substituting the estimate  (\ref{eq:BT:imp:energy:scal}) for $Z^2=\pa\pa$ into (\ref{1.27.4.18}), also by using Lemma \ref{1.6.4.18},  we  obtain the following energy flux estimate
\begin{equation*}
E[D D_Z\phi](\H_{u_1}^{-u_2})+E[D D_Z\phi](\Hb_{-u_2}^{u_1})\les (\mathcal{E}_{2, \ga_0}+\dn^2) (u_1)_+^{-\ga_0+2\zeta(\pa Z)}, \forall Z=\Omega_{ij}, \Omega_{0j}.
\end{equation*}
Thus we proved (\ref{eq:BT:imp:energy:scal}) for $Z^2=\pa Z,\, Z\in \Ga$. Substituting the result back to Lemma \ref{prop:Q:need:r:all} yields
\begin{align*}
\iint_{\mathcal{D}_{u_1}^{-u_2}} &r^{1+\ep}u_+^{1+\ep}|(\Box_A-1)D_Z^k\phi|^2 \les (\E_{2,\ga_0}+ \dn^2)(u_1)_+^{2\zeta(Z^k)+1-\ga_0}\\
&+q_0^2 (\E_{2,\ga_0}+\dn^2)\int_{u_2}^{u_1} u_+^{2\zeta(Z^k)+2\ep-\ga_0-2} du
\end{align*}
for any  $Z^k \in \Ga^k, k=2$. By direct integration, we can obtain Proposition \ref{1.27.3.18}.

 (\ref{eq:BT:imp:energy:scal}) with  the general $Z^k\in \Ga^k$ can be proved by applying (\ref{1.27.2.18})  with the help of (\ref{1.22.1.17}) and Lemma \ref{1.6.4.18}.

To prove (\ref{eq:BT:imp:pwe:scal}), we  apply  the energy identity \eqref{pWescaout} with $(\phi, \G)=(D_Z^k\phi, 0)$  to derive
\begin{equation}\label{12.7.7.17}
\begin{split}
\int_{\H_{u_1}^{-u_2}}&r^{-1}|D_L(rD_Z^k\phi)|^2+\int_{\Hb_{-u_2}^{u_1} }r(|\D D_Z^k\phi|^2+|D_Z^k\phi|^2)\\
\les &\mathcal{E}_{k, \ga_0} (u_1)_+^{1-\ga_0+2\zeta(Z^k)}+\iint_{\mathcal{D}_{u_1}^{-u_2}}|D_Z^k\phi|^2\\
&+\iint_{\mathcal{D}_{u_1}^{-u_2}} (r |F_{L\mu}|| J[D_Z^k \phi]^\mu|+|(\Box_A-1)D_Z^k\phi| |D_L(rD_Z^k\phi)|) .
\end{split}
\end{equation}
 The second term is estimated by using \eqref{eq:BT:imp:energy:scal} as follows
\begin{align*}
\iint_{\mathcal{D}_{u_1}^{-u_2}} |D_Z^k\phi|^2 \les (\mathcal{E}_{k, \ga_0} +\dn^2)(u_1)_+^{1-\ga_0+2\zeta(Z^k)},\, \,  k\le 2.
\end{align*}
Next we treat the nonlinear terms on the right hand side of (\ref{12.7.7.17}).  Note that
$
r^{2}J[D_Z^k\phi]=J[rD_Z^k\phi]$,
 and recall from Proposition \ref{prop:decay:pt:Max} that
$$
 |\rho|^2\les (\dn+q_0^2) r^{-2}u_+^{-2},\quad |\a|^2\les \dn r^{-2-\ga_0}u_+^{-1}.
$$
We then can bound that
\begin{align*}
&r |F_{L\mu}|| J[D_Z^k \phi]^\mu|+|(\Box_A-1)D_Z^k\phi| |D_L(rD_Z^k\phi)|\\
&\les r^{-1}u_+^{-1-\ep}|D_L(rD_Z^k\phi)|^2+  ru_+^{1+\ep}(|(\Box_A-1)D_Z^k\phi|^2+|\rho|^2|D_Z^k\phi|^2)\\
&\qquad +r^{-2-\ep}|\D(rD_Z^k\phi)|^2+r^{2+\ep}|\a|^2|D_Z^k\phi|^2\\
& \les r^{-1}u_+^{-1-\ep}|D_L(rD_Z^k\phi)|^2+ r^{-2-\ep}|\D(rD_Z^k\phi)|^2+ru_+^{1+\ep}|(\Box_A-1)D_Z^k\phi|^2+(\dn+q_0^2) u_+^{-2+\ep} |D_Z^k\phi|^2.
\end{align*}
 The integral of the the first two terms can be bounded by the weighted fluxes on $\H_{u_1}^{-u_2}$ and $\Hb_{-u_2}^{u_1}$. The integral of the last term can be bounded by using (\ref{eq:BT:imp:energy:scal}) as follows,
\begin{align*}
\iint_{\mathcal{D}_{u_1}^{-u_2}} (\dn+q_0^2) u_+^{-2+\ep} |D_Z^k\phi|^2 \les (\E_{k,\ga_0}+\dn^2) (u_1)_+^{2\zeta(Z^k)-1+\ep-\ga_0}.
\end{align*}
The commutator can be treated by using (\ref{1.22.1.17}). Combining all these estimates, by using Lemma \ref{1.6.4.18},  we have
$$
\int_{\H_{u_1}^{-u_2}}r^{-1} |D_L(rD_Z^k\phi)|^2+\int_{\Hb_{-u_2}^{u_1} }r(|\D D_Z^k\phi|^2+|D_Z^k\phi|^2) \les (\mathcal{E}_{k, \ga_0} +\dn^2) (u_1)_+^{1-\ga_0+2\zeta(Z^k)}.
$$
Note that $\int_{\H_{u_1}^{-u_2}} r^{-1} |D_Z^k \phi|^2 $ can be treated by  (\ref{eq:BT:imp:energy:scal}).
Thus we can obtain the $r$-weighted energy flux decay \eqref{eq:BT:imp:pwe:scal} for the scalar field.

\end{proof}

\subsection{Improving the bootstrap assumptions}
Let $C$ be the implicit constants in Proposition \ref{prop:BT:imp:Max} and Proposition \ref{prop:BT:imp:scal}. Without loss of generality we may assume $C\geq 2$. This ensures that $\dn\geq \mathcal{E}_{2, \ga_0}$. By our convention, the constant $C$ depends only on $\ep$, $\ga_0$ and $|q_0|$. Let $\mathcal{E}_{2, \ga_0}$, $\dn$ verify the following conditions:
\[
C\mathcal{E}_{2, \ga_0}= \f12 \dn,\quad C\dn\leq \f12,
\]
that is,
\[
\mathcal{E}_{2, \ga_0}\leq \frac{1}{4C^2}.
\]
Then we  have
\[
C(\mathcal{E}_{2, \ga_0}+\dn^2) \leq \dn.
\]
We thus have improved the bootstrap assumptions \eqref{eq:BT:energy}-\eqref{eq:BT:PWE:Max}.

Thus  in view of Proposition \ref{prop:decay:scal} and Proposition \ref{prop:decay:pt:Max},  we can obtain the pointwise estimates in Theorem \ref{m.thm}  part (1). In view of Proposition \ref{prop:BT:imp:Max} and Proposition \ref{prop:BT:imp:scal}, we can obtain the  set of estimates on energy fluxes in Theorem \ref{m.thm}   part (2).
%\begin{thm}\label{thm2}
%There exists a small constant $\ep_0>0$, if $\E_{2,\ga_0}\le \ep_0$,  for the given admissible data, there exists a unique solution $(F, \phi)$ for %(\ref{EQMKG}) and there holds for the solution $(F, \phi)$  that
%\begin{align*}
%&E[D_Z^k\phi, \cL_Z^k\tilde{F}](\H_{u_1}^{-u_2})+E[D_Z^k\phi, \cL_Z^k\tilde{F}](\Hb_{-u_2}^{u_1})< C (u_1)_+^{-\ga_0+2\zeta(Z^k)}\ep_0,\\
%&\int_{\H_{u_1}^{-u_2}}r|D_L D_Z^k\phi|^2 r^2dvd\om+\int_{\Hb_{-u_2}^{u_1} }r(|\D D_Z^k\phi|^2+|D_Z^k\phi|^2) r^2dud\om < %C(u_1)_+^{1-\ga_0+2\zeta(Z^k)}\ep_0,\\
%&\int_{\H_{u_1}^{-u_2}}r^{\ga_0}|\a[\cL_Z^k \tilde{F}]|^2 r^2dvd\om+\iint_{\mathcal{D}_{u_1}^{-u_2}}r^{\ga_0-1}|\a, \rho,\sigma[\cL_Z^k\tilde{F}]|^2 %r^2dudvd\om\\
%&\qquad \qquad+\int_{\Hb_{-u_2}^{u_1} }r^{\ga_0}(|\rho[\cL_Z^k\tilde{F}]|^2+|\si[\cL_Z^k\tilde{F}]|^2) r^2dud\om <C (u_1)_+^{2\zeta(Z^k)}\ep_0
%\end{align*}
%for all $u_2<u_1\leq -\frac{R}{2}$, $Z^k=Z_1 Z_2\ldots Z_k$ with $Z_i\in \Gamma$, $k\leq 2$, where $C>0$ is a universal constant.
%\end{thm}

\end{document}